\documentclass[11pt, reqno]{amsart}

\usepackage{amsmath,amssymb,amscd,mathrsfs,amscd}
\usepackage{graphics,verbatim, tikz}

\usetikzlibrary{matrix,arrows,patterns}
\usepackage{amsmath,amstext,amsbsy,amssymb,amscd}
\usepackage{amsmath}
\usepackage{amsxtra}
\usepackage{amscd}
\usepackage{amsthm}
\usepackage{amsfonts}
\usepackage{graphicx}%
\usepackage{amssymb}
\usepackage{eucal}
\usepackage{color}
\usepackage{multirow}
\usepackage[enableskew,vcentermath]{youngtab}
\usepackage{hyperref}
\hypersetup{colorlinks,linkcolor=blue,urlcolor=cyan,citecolor=blue}

\newtheorem{thm}{Theorem}[section]
\theoremstyle{plain}
\newtheorem{lem}[thm]{Lemma}
\newtheorem{prop}[thm]{Proposition}

\newtheorem{cor}[thm]{Corollary}

\theoremstyle{definition}

\theoremstyle{remark}
\newtheorem{rem}[thm]{Remark}

\numberwithin{equation}{section}

\newcommand{\A}{\CMcal A}
\newcommand{\Ap}{{\mathcal A^\pi}}

\newcommand{\la}{\lambda}

\newcommand{\osp}{\mathfrak{osp}}
\newcommand{\sq}{\sqrt{\pi}}
\newcommand{\qq}{{\bf q}}

\newcommand{\tha}{\theta}

\newcommand{\vep}{\varepsilon}
\newcommand{\vv}{{\bf v}}
\newcommand{\wtd}{\widetilde}

\newcommand{\Fr}{{\rm Fr}}
\newcommand{\N}{\mathbb N}

{\vskip-\lastskip\medskip
  \noindent
  {\em #1.}\enspace
  }%
{\qed\par\medskip
  }
\newcommand{\Q}{{\mathbb{Q}}}
\newcommand{\Z}{{\mathbb Z}}

\newcommand{\curlyA}{{{\A}}}

\newcommand{\fA}{{}_\A{\mathbf{f}}}
\newcommand{\fd}{{\bf f}^\diamond}
 
\newcommand{\fR}{{}_R{\mathbf{f}}}
\newcommand{\fRd}{{}_R{\mathbf{f}^\diamond}} 
\newcommand{\kf}{{\mathfrak f}}

\newcommand{\UU}{{\bf U}}

\newcommand{\AU}{\vphantom{|}_{\curlyA}{\bf U}}

\newcommand{\UUdot}{\dot{\bold{U}}}
\newcommand{\UUAdot}{{}_\A\dot{\bold{U}}}
\newcommand{\UURdot}{{}_{R}\dot{\bold{U}}}

\newcommand{\bbinom}[2]{\begin{bmatrix}#1 \\ #2\end{bmatrix}}

\newcommand{\zero}{{\bar{0}}}
\newcommand{\one}{{\bar{1}}}

\newcommand{\ang}[1]{\left\langle#1\right\rangle}
\newcommand{\bra}[1]{\left[#1\right]}

\renewcommand{\bar}[1]{\overline{#1}}

\begin{document}
\title[Quantum supergroups at roots of 1]{Quantum  supergroups VI.  Roots of $1$} 

\author[Christopher Chung]{Christopher Chung}
\author[Thomas Sale]{Thomas Sale}
\author[Weiqiang Wang]{Weiqiang Wang}
\address{Department of Mathematics, University of Virginia, Charlottesville, VA 22904}
\email{cc2wn@virginia.edu (Chung), tws2mb@virginia.edu (Sale), ww9c@virginia.edu (Wang)}


\begin{abstract}
A quantum covering group is an algebra with parameters $q$ and $\pi$ subject to $\pi^2=1$ and it admits an integral form; it specializes to the usual quantum group at $\pi=1$ and to a quantum supergroup of anisotropic type at $\pi=-1$. In this paper we establish the Frobenius-Lusztig homomorphism and Lusztig-Steinberg tensor product theorem in the setting of quantum covering groups at roots of 1. The specialization of these constructions at $\pi=1$ recovers Lusztig's constructions for quantum groups at roots of 1. 
\end{abstract}


\maketitle

\setcounter{tocdepth}{1}
 \tableofcontents

\section{Introduction}

\subsection{}

A Drinfeld-Jimbo quantum group with the quantum parameter $q$ admits an integral $\Z[q,q^{-1}]$-form; its specialization at $q$ being a root of 1 were studied by Lusztig in \cite{Lu90a, Lu90b}, \cite[Part V]{Lu94} and also by many other authors. In these works Lusztig developed the quantum group version of Frobenius homomorphism and Frobenius kernel (known as small quantum groups), as a quantum analogue of several classical concepts arising from algebraic groups in a prime characteristic. The quantum groups at roots of 1 and their representation theory form a substantial part of Lusztig's program on modular representation theory, and they have further impacted other areas including geometric representation theory and categorification. 

A quantum covering group $\UU$, which was introduced in \cite{CHW1} (cf. \cite{HW15}),  is an algebra defined via super Cartan datum, which depends on parameters $q$ and $\pi$ subject to $\pi^2=1$. A quantum covering group specializes at $\pi=1$ to a quantum group and at $\pi=-1$ to a quantum supergroup of anisotropic type (see \cite{BKM98}). Half the quantum covering group with parameter $\pi$ with $\pi^2=1$ appeared first in \cite{HW15} in an attempt to clarify the puzzle why quantum groups are categorified once more by the (spin) quiver Hecke superalgebras introduced in \cite{KKT16}. There has been much further progress on odd/spin/super categorification of quantum covering groups; see \cite{KKO14, EL16, BE17}. 

For quantum covering groups, the $(q,\pi)$-integer
$$
[n]_{q,\pi} =\frac{(\pi q)^n-q^{-n}}{\pi q - q^{-1}}\in \N[q,q^{-1},\pi]
$$ 
and the corresponding $(q,\pi)$-binomial coefficients are used, and they help to restore the positivity which is lost in the quantum supergroup with $\pi=-1$. The algebra $\UU$ (and its modified form $\UUdot$, respectively) admits an integral $\Z[q,q^{-1},\pi]$-form $\AU$ (and $\UUAdot$, respectively). In \cite{CHW2} and then \cite{Cl14} the canonical bases arising from quantum covering groups \`a la Luszitig and Kashiwara were constructed, and this provided for the first time a systematic construction of canonical bases for quantum supergroups. The braid group action has been constructed in \cite{CH16} for quantum covering groups, and the first step toward a geometric realization of quantum covering groups was taken in \cite{FL15}.

\subsection{}

To date the main parts of the book of Lusztig \cite{Lu94} have been generalized to the quantum covering group setting, except part V on roots of 1 and Part II on geometric realization in full generality. The goal of this paper is to fill a gap in this direction by presenting a systematic study of the quantum covering groups at roots of 1; we follow closely the blueprint in \cite[Chapters 33--36]{Lu94}.  

\subsection{}

We impose a mild {\em bar-consistent} assumption on the super Cartan datum in this paper, following \cite{HW15, CHW2}. This assumption ensures that the new super Cartain datum and root datum arising from considerations of roots of 1 work as smoothly as one hopes. The assumption turns out to be also most appropriate again for the existence of Frobenius-Lusztig homomorphisms for quantum covering groups. 

We expect that the quantum covering groups of finite type at roots of 1 have very interesting representation theory, which has yet to be developed (compare \cite{AJS94}). 
The categorification of the quantum covering group {\em of rank one} at roots of 1 is already highly nontrivial as shown in the recent work of Egilmez and Lauda \cite{EgL18}. We hope our work on higher rank quantum covering groups could provide a solid algebraic foundation for further super categorification and connection to quantum topology. 

Specializing at $\pi=-1$, we obtain the corresponding results for (half, modified) quantum supergroups of anisotropic type at roots of 1; this class of quantum supergroups includes the quantum supergroup  of type $\mathfrak{osp}(1|2n)$ as the only finite type example. It will be very interesting to develop systematically the quantum supergroups at roots of 1 associated to the {\em basic} Lie superalgebras (i.e., the simple Lie superalgebras with non-degenerate supersymmetric bilinear forms). 
\subsection{}

Below we provide some more detailed descriptions of the results and the organization of the paper. 
In Section~\ref{sec:binom}, we establish several basic properties of the $(q,\pi)$-binomial coefficients at roots of 1, generalizing Lusztig \cite[Chapter 34]{Lu94}. 

In Section~\ref{sec:QCG}, we recall half the quantum covering group $\fR$ and the whole (respectively, the modified) quantum covering group $\UU$ (respectively, ${}_{R}\dot{\bold{U}}$) over some ring $R^\pi$, associated to a super Cartan datum. We give a presentation of ${}_{R}\dot{\bold{U}}$ and a presentation of the quasi-classical counterpart  $\fRd$ of $\fR$, generalizing  \cite[33.2]{Lu94}. 

Our Section~\ref{sec:Frob} is a generalization of \cite[Chapter 35]{Lu94}. We establish in Theorem~\ref{th:Frob'} a $R^\pi$-superalgebra homomorphism 
$\Fr': \fRd \longrightarrow \fR$, which sends the generators $\theta_i^{(n)}$ to $\theta_i^{(n\ell_i)}$ for all $i\in I, n$.
This is followed by the Lusztig-Steinberg tensor product theorem for $\fR$ which we prove in Theorem~\ref{thm:tensor}. 
Next we establish in Theorem~\ref{th:Frob} the Frobenius-Lusztig homomorphism $\Fr: \fR \longrightarrow \fRd$
which sends the generators $\theta_i^{(n)}$ to $\theta_i^{(n/\ell_i)}$ if $\ell_i$ divides $n$, and to $0$ otherwise, for all $i\in I, n$. We further extend the homomorphism $\Fr$ to the modified quantum covering group in Theorem~\ref{th:Frob2}. 

Finally in Section~\ref{sec:small}, we formulate the small quantum covering groups and show it is a Hopf algebra. In case of finite type (i.e., corresponding to $\mathfrak{osp}(1|2n)$ or $\mathfrak{so}(1+2n)$), we show that the small quantum covering group is finite dimensional.

\subsection*{Acknowledgement} 
This research is partially supported by Wang's NSF grant DMS-1702254 (including GRA supports for the two junior authors).  WW thanks Adacemia Sinica Institute of Mathematics (Taipei) for the hospitality and support during a past visit, where some of the work was carried out.

\section{The $(q,\pi)$-binomials at roots of $1$}
  \label{sec:binom}

In this section, we establish several basic formulas of the $(q,\pi)$-binomial coefficients at roots of 1. They specialize to the formulas in \cite[Chapter 34]{Lu94} at $\pi=1$.

\subsection{}

Let $\pi$ and $q$ be formal indeterminants such that $\pi^2=1$. Fix $\sq$ such that $\sq^2 =\pi$. In contrast to earlier papers on the quantum covering groups \cite{CHW1, CHW2, CFLW, Cl14}, it is often helpful and sometimes crucial   for the ground rings considered in this paper to contain $\sq$, and for the sake of simplicity we choose to do so uniformly from the outset.  For any ring $S$ with 1, define the new ring
\[
S^{\pi} = S\otimes_\Z \Z[\sq].
\]
We shall use often the following two rings:
\[
\A =\Z[q,q^{-1}], \qquad \Ap =\Z[q,q^{-1},\sq].
\] 

Let $\N =\{0,1,2,\ldots\}$. For $a\in\Z$ and $n\in\N$, we define the  {\em $(q,\pi)$-integer}
\begin{equation*}
 [a]_{q,\pi} =\frac{(\pi q)^a-q^{-a}}{\pi q - q^{-1}}\in \Ap,
\end{equation*}
and then define the corresponding $(q,\pi)$-factorials and $(q,\pi)$-binomial coefficients
by
\begin{equation*}
[n]_{q,\pi}^! =\prod_{i=1}^n [i]_{q,\pi},\qquad
\bbinom{a}{n}_{q,\pi} =\frac{\prod_{i=1}^n[a+1-i]_{q,\pi}}{[n]_{q,\pi}^!}.
\end{equation*}
For an indeterminant $v$, we denote the $v$-integers
\begin{equation*}
 [a]_v=\frac{v^a-v^{-a}}{v - v^{-1}}
\end{equation*}
and we similarly define the $v$-factorials $[n]_v^!$ and $v$-binomial coefficients $\bbinom{a}{n}_v$.
We denote by $\binom{a}{n}$ the classical binomial coefficients.

\subsection{}  

In this paper, the notation $v$ is auxiliary, and we will identify
$$v:=\sq  q,
$$
and hence, for $n,t \in \N$, 
\begin{align}
  \label{eq:vq}
\begin{split}
 [n]_{q,\pi}  =\sq^{n-1}[n]_v, & \qquad
 [n]_{q,\pi}^!  =\sq^{n(n-1)/2} [n]_v^!,
   \\ \\
 \bbinom{n}{t}_{q,\pi} &=\sq^{(n-t)t} \bbinom{n}{t}_v.
 \end{split}
\end{align}

\subsection{}  
\label{subsec:root1}

Fix $ \ell \in \Z_{>0}$ and let $\ell' =\ell$ or $2\ell$ if $\ell$ is odd and let $\ell' = 2\ell$ if $\ell$ is even. Let 
\[
{\A}' = {\A}/\langle f(q) \rangle, 
\]
where ${\A}/\langle f(q) \rangle$ denotes the ideal generated by the $\ell'$-th cyclotomic polynomial $f(q)$; we denote by $\vep \in \A'$ the image of $q \in \A$.  Take $R$ to be an ${\A}'$-algebra with 1 (and so also an ${\A}$-algebra).
Introduce the following root of 1 in $R^\pi$: 
\begin{equation}
  \label{eq:qq}
\qq = \sq \vep \in R^{\pi}. 
\end{equation}
Then the element
\[
\vv :=\sq \qq \in R^{\pi}
\] 
satisfies that 
\begin{equation}
   \label{eq:order}
\vv^{2\ell} =1, \qquad \vv^{2t}\neq 1 \quad (\forall t \in \Z, \ell >t  > 0).
\end{equation}
Consider the specialization homomorphism  $\phi: \A^{\pi} \rightarrow R^{\pi}$
which sends $q$ to $\qq$ and $\sq$ to $\sq$. We shall denote by
$[n]_{\qq,\pi}$ and $\bbinom{n}{t}_{\qq,\pi}$ the images of $[n]_{q,\pi}$ and  $\bbinom{n}{t}_{q,\pi}$ under $\phi$ respectively, and so on.

The following lemma is an analogue of \cite[Lemma 34.1.2]{Lu94}, which can be in turn recovered by setting $\pi=1$ below.

\begin{lem}
  \label{lem:34.1.2}
\begin{enumerate}
\item[(a)] If $t \in \Z_{>0}$ is not divisible by $\ell$ and $n\in \Z$ is divisible by $\ell$, then $ \bbinom{n}{t}_{\qq,\pi} =0.$

\item[(b)]
If $n_1\in\Z$ and $t_1\in \N$, then we have 
$$\bbinom{\ell n_1}{\ell t_1}_{\qq,\pi} =\pi^{\ell^2 t_1 (n_1- (t_1-1)/2)} \qq^{\ell^2t_1(n_1+1)} \binom{n_1}{t_1}.
$$

\item[(c)]
Let $n\in \Z$ and $t\in \N$. Write $n=n_0+\ell n_1$ with $n_0, n_1\in \Z$ such that $0\le n_0 \le \ell-1$ and write 
$t=t_0+\ell t_1$ with $t_0, t_1\in \N$ such that $0\le t_0 \le \ell-1$.  Then we have
 $$
 \bbinom{n}{t}_{\qq,\pi} 
 =  \pi^{\ell (n_0-t_0)t_1 +     \ell^2(n_1 -(t_1-1)/2 )t_1} 
  \qq^{\ell(n_0t_1 -n_1 t_0) + \ell^2(n_1+1)t_1}    \bbinom{n_0}{t_0}_{\qq,\pi} \binom{n_1}{t_1}.
 $$
\end{enumerate}
\end{lem}

\begin{proof}
One proof would be by imitating the arguments for  \cite[Lemma 34.1.2]{Lu94}.
Below we shall use an alternative and quicker approach, which is to convert \cite[Lemma 34.1.2]{Lu94} into our current statements using \eqref{eq:vq} via the substitution $\vv =\sq \qq$. 
 Part (a) immediately follows from \cite[Lemma 34.1.2(a)]{Lu94}.
  
 (b) By applying \cite[Lemma 34.1.2(b)]{Lu94} to $\bbinom{\ell n_1}{\ell t_1}_\vv$ and using \eqref{eq:vq}, we have
 \begin{align*}
 \bbinom{\ell n_1}{\ell t_1}_{\qq,\pi} &= \sq^{\ell t_1(\ell n_1 -\ell t_1)} \bbinom{\ell n_1}{\ell t_1}_\vv
 =\sq^{\ell^2 t_1(n_1 - t_1)} \vv^{\ell^2 t_1(n_1 +1)}  \binom{n_1}{t_1},
 \end{align*}
 which can be easily shown to be equal to the formula as stated in the lemma. 
 
 (c) Note that   
 \begin{align}  \label{eq:pipower}
\sq^{(n-t)t}  
 =   \sq^{\ell ((n_0-t_0)t_1 + (n_1 -t_1)t_0)} \sq^{\ell^2(n_1 -t_1)t_1} \sq^{(n_0-t_0)t_0}.
 \end{align}
 By applying \cite[Lemma 34.1.2(c)]{Lu94} to $\bbinom{n}{t}_\vv$ and using \eqref{eq:vq}-\eqref{eq:pipower}, we have
 \begin{align*}
 \bbinom{n}{t}_{\qq,\pi}
  &= \sq^{(n-t)t}  \bbinom{n}{t}_\vv
    \\
   &= \sq^{(n-t)t}  \vv^{\ell (n_0t_1 -n_1t_0) +\ell^2(n_1+1)t_1} \bbinom{n_0}{t_0}_\vv \binom{n_1}{t_1}
    \\
   &= \sq^{\ell ((n_0-t_0)t_1 + (n_1 -t_1)t_0)} \sq^{\ell^2(n_1 -t_1)t_1}
    \sq^{\ell (n_0t_1 -n_1t_0) +\ell^2(n_1+1)t_1} 
    \\
    &\qquad \times \qq^{\ell (n_0t_1 -n_1t_0) +\ell^2(n_1+1)t_1} 
    \left(\sq^{(n_0-t_0)t_0}  \bbinom{n_0}{t_0}_\vv \right)  \binom{n_1}{t_1}
     \\
   &= \pi^{\ell (n_0-t_0)t_1 +     \ell^2(n_1 -(t_1-1)/2 )t_1} 
     \qq^{\ell (n_0t_1 -n_1t_0) +\ell^2(n_1+1)t_1}  \bbinom{n_0}{t_0}_{\qq,\pi} \binom{n_1}{t_1}.
\end{align*} 
The lemma is proved. 
\end{proof}

 Note that, due to our choice of $\qq =\sqrt{\pi} \vep$, we also have an analogue of equation~ (e) in the proof of \cite[Lemma 34.1.2]{Lu94}:
\begin{equation}   \label{eq:34.1.2e}
\vv^{\ell^2 + \ell} = \pi^{(\ell+1)\ell/2} \qq^{\ell^2+\ell} =(-1)^{\ell+1}.
\end{equation}

\subsection{}

The following is an analogue of \cite[\S 34.1.3(a)]{Lu94}.

\begin{lem}
 \label{lem:34.1.3}
Let $b \ge 0$. Then 
$$
\frac{[\ell b]_{\qq,\pi}^!}{([\ell]_{\qq,\pi}^!)^b} =b! (\pi\qq)^{\ell^2 b(b-1)/2}.
$$
\end{lem}

\begin{proof}
Recall $\vv =\sq \qq$.
Using \eqref{eq:vq} and \cite[\S 34.1.3(a)]{Lu94}, we have
 \begin{align*}
[\ell b]_{\qq,\pi}^!/ ([\ell]_{\qq,\pi}^!)^b
 &= \sq^{ \ell b (\ell b-1)/2 -b \ell(\ell-1)/2}  [\ell b]_\vv^!/ ([\ell]_\vv^!)^b
 \\
 &= \sq^{ \ell^2 b (b-1)/2} b! \vv^{\ell^2 b(b-1)/2} 
 =b! (\pi\qq)^{\ell^2 b(b-1)/2}.
\end{align*} 
The lemma is proved. 
\end{proof}

Below is a $\pi$-enhanced version of \cite[Lemma 34.1.4]{Lu94}.

\begin{lem}
\label{lem:34.1.4}
Suppose that $0 \leq r \leq a < \ell.$ Then,
\[
\sum_{s = 0}^{\ell - a - 1} (-1)^{\ell - r + 1 + s}\pi^{\binom{s+1}{2} + s(r-\ell)}\qq^{-(\ell - r)(a - \ell + 1 + s) + s}\bbinom{\ell - r}{s}_{\qq,\pi} = \pi^{\binom{r}{2} - \binom{l}{2} - a(r - l)}\qq^{\ell(a - r)}\bbinom{a}{r}_{\qq,\pi}.
\] 
\end{lem}

\begin{proof}
Plugging $\vv = \sq \qq$ into \cite[Lemma 34.1.4]{Lu94} and using \eqref{eq:vq}, we obtain
\[
\begin{split}
\sum_{s = 0}^{\ell - a - 1} (-1)^{\ell - r + 1 + s}\sq^{-(\ell - r)(a - \ell + 1 + s) + s +s(s - \ell + r)} & \qq^{-(\ell - r)(a - \ell + 1 + s) + s}\bbinom{\ell - r}{s}_{\qq,\pi} \\
&= \sq^{\ell(a - r) + r(r - a)}\qq^{\ell(a - r)}\bbinom{a}{r}_{\qq,\pi}.
\end{split}
\]
Rearranging the $\sq$ terms, we have
\begin{align*}
\sum_{s = 0}^{\ell - a - 1} & (-1)^{\ell - r + 1 + s}\sq^{s(s+1) + 2s(r-\ell)} \qq^{-(\ell - r)(a - \ell + 1 + s) + s}\bbinom{\ell - r}{s}_{\qq,\pi} \\
&= \sq^{r(r-1) -\ell(\ell - 1) - 2a(r - l)}\qq^{\ell(a - r)}\bbinom{a}{r}_{\qq,\pi}.
\end{align*}
from which the desired formula is immediate.
\end{proof}


\section{Quantum covering groups at roots of 1}
  \label{sec:QCG}
  
In this section we recall the notion of super Cartan/root datum and the quantum covering groups. Then we obtain presentations of the modified quantum covering groups and their quasi-classical counterpart. 

\subsection{} 
  \label{subsec:data}

The following is an analogue of \cite[\S2.2.4-5]{Lu94}. 

A {\em Cartan datum} is a pair $(I,\cdot)$ consisting of a finite
set $I$ and a symmetric bilinear form $\nu,\nu'\mapsto \nu\cdot\nu'$
on the free abelian group $\Z[I]$ with values in $\Z$ satisfying
\begin{enumerate}
 \item[(a)] $d_i= \frac{i\cdot i}{2}\in \Z_{>0}$;

  \item[(b)]
$2\frac{i\cdot j}{i\cdot i}\in -\N$ for $i\neq j$ in $I$.
\end{enumerate}
If the datum can be decomposed as $I=I_0 \coprod I_1$ such that
\begin{enumerate}
        \item[(c)] $I_1 \neq\emptyset$,
        \item[(d)] $2\frac{i\cdot j}{i\cdot i} \in 2\Z$ if $i\in I_1$,
\end{enumerate}
then it is called a {\em super Cartan datum}; cf. \cite{CHW1}. We denote the parity $p(i)=0$ for $i\in I_0$ and $p(i)=1$ for $i\in I_1$. 

Following \cite{CHW1}, we will always assume a super Cartan datum satisfies the additional  {\em bar-consistent} condition:
\begin{enumerate}
        \item[(e)] $\frac{i\cdot i}{2} \equiv p(i) \mod 2, \quad \forall i\in I.$
\end{enumerate}

A root datum of type $(I, \cdot)$ consists of 2 finite rank lattices $X, Y$ with a perfect bilinear pairing $\langle \cdot, \cdot \rangle: Y \times X \rightarrow \Z$, 2 embeddings $I \hookrightarrow X$ $(i \mapsto i')$ and $I \hookrightarrow Y$ $(i \mapsto i)$ such that $\langle i, j' \rangle =2\frac{i\cdot j}{i\cdot i}$, $\forall i,j\in I$. 
Moreover, we will assume throughout the paper that the root datum is {\em $X$-regular}, i.e., that the simple roots are linearly independent in $X$.

Define
\[
\ell_i = \min \{ r \in \Z_{>0} \mid r (i\cdot i)/2 \in \ell \Z\}.
\]
The next lemma follows by the definition of $\ell_i$ and the bar-consistency condition of $I$.

\begin{lem}
 \label{lem:parity1}
For each $i\in I_1$,   $\ell_i$ has the same parity as $\ell$.
\end{lem}

Then $(I,\diamond)$ is a new root datum by \cite[2.2.4]{Lu94}, where we let 
\[
i\diamond j =(i\cdot j) \ell_i \ell_j, \quad \forall i,j \in I.
\]
Note that if $\ell$ is odd, then $(I,\diamond)$ is a super Cartan datum with the same parity 
decomposition $I =I_0 \cup I_1$ as for $(I,\cdot)$ by Lemma~\ref{lem:parity1};
if $\ell$ is even, then $(I,\diamond)$ is a (non-super) Cartan datum with $I_1 =\emptyset$.

We shall write $Y^\diamond, X^\diamond$ in this paper what Lusztig \cite[2.2.5]{Lu94}  
denoted by $Y^*, X^*$ respectively, and we will use superscript $^\diamond$
in related notation associated to $(Y^\diamond, X^\diamond, I, \diamond)$ below. More explicitly, we set $X^\diamond =\{\zeta \in X | \langle i, \zeta \rangle \in \ell_i \Z, \forall i \in I\}$ and $Y^\diamond =\text{Hom}_\Z (X^\diamond, \Z)$ with the obvious pairing. The embedding $I \hookrightarrow X^\diamond$ is given by $i \mapsto i'^\diamond =\ell_i i'\in X$, while embedding $I \hookrightarrow Y^\diamond$ is given by $i \mapsto i^\diamond \in Y^\diamond$ whose value at any $\zeta \in X^\diamond$ is $\langle i, \zeta \rangle / \ell_i$. It follows that $\langle i^\diamond, j'^\diamond \rangle = 2 i \diamond j /i\diamond i$.

If $\ell$ is odd, then  $(Y^\diamond, X^\diamond, \cdots)$ is a new super root datum satisfying (a)-(d) above and in addition the bar-consistency condition (e).
Indeed, we have $2\frac{i\diamond j}{i\diamond i} =2\frac{i\cdot j}{i\cdot i} \frac{\ell_j}{\ell_i} \in 2\Z$ by Lemma~\ref{lem:parity1}, whence (d), and  $\frac{i\diamond i}{2} = \frac{i\cdot i}{2} \ell_i^2 \equiv p(i) \mod 2$ by Lemma~\ref{lem:parity1}, whence (e).
If $\ell$ is even, then $(Y^\diamond, X^\diamond, \cdots)$ is a new (non-super) root datum just as in \cite[2.2.5]{Lu94}.

\subsection{}


By \cite[Propositions 1.4.1, 3.4.1]{CHW1},  the unital $\mathbb{Q}(q)^{\pi}$-superalgebra $\bold{f}$ is generated by 
$\theta_i$ $(i\in I)$ subject to the super Serre relations
\[
\sum_{n+n'=1-\ang{i,j'}}(-1)^{n'}\pi_i^{n'p(j)+\binom{n'}{2}}
\theta_i^{(n)}\theta_j\theta_i^{(n')}=0
\]
for any $i\neq j$ in $I$; here a generator $\theta_i$ is even if and only if $i\in I_0$. 
There is an ${\A}^{\pi}$-form for $\bold{f}$, which we call ${}_{{\A}}\bold{f}$. It is generated by the divided powers $\theta_{i}^{(n)} = \theta_{i}^{n}/[n]_{q_{i}, \pi_{i}}^{!}$ for all $i\in I, n\ge 1.$  As $R^{\pi}$ is an ${\A}^{\pi}$-algebra (cf. \S \ref{subsec:root1}), by a base change we define ${}_{R}\bold{f} = R^{\pi} \otimes_{{\A}^{\pi}} {}_{{\A}}\bold{f}.$ The algebras $'\bold{f}^{\diamond}$, $\bold{f}^{\diamond}$ and ${}_{R}\bold{f}^{\diamond}$ are defined in the same way using the Cartan datum $(I, \diamond).$

Let $\bold{U}$ denote the quantum covering group associated to the root datum $(Y, X, . . .)$ introduced in \cite{CHW1}. By \cite[Proposition 3.4.2]{CHW1}, $\UU$ is a unital $\mathbb{Q}(q)^{\pi}$-superalgebra with generators
\[\
E_i\quad(i\in I),\quad F_i\quad (i\in I), \quad J_{\mu}\quad (\mu\in
Y),\quad K_\mu\quad(\mu\in Y),
 \]
subject to the relations (a)-(f) below for
all $i, j \in I, \mu, \mu'\in Y$:
\[
\tag{a} K_0=1,\quad K_\mu K_{\mu'}=K_{\mu+\mu'},
\]
\[
J_{2\mu}=1, \quad J_\mu J_{\mu'}=J_{\mu+\mu'},
\quad
J_\mu K_{\mu'}=K_{\mu'}J_{\mu},
\]
\[
\tag{b} K_\mu E_i=q^{\ang{\mu,i'}}E_iK_{\mu}, \quad
J_{\mu}E_i=\pi^{\ang{\mu,i'}} E_iJ_{\mu},
\]
\[
\tag{c} \; K_\mu F_i=q^{-\ang{\mu,i'}}F_iK_{\mu}, \quad J_{\mu}F_i=
\pi^{-\ang{\mu,i'}} F_iJ_{\mu},
\]
\[
\tag{d} E_iF_j-\pi^{p(i)p(j)}
F_jE_i=\delta_{i,j}\frac{\wtd{J}_{i}\wtd{K}_i-\wtd{K}_{-i}}{\pi_iq_i-q_i^{-1}},
\]
\[
\tag{e}
\sum_{n+n'=1-\ang{i,j'}}(-1)^{n'}\pi_i^{n'p(j)+\binom{n'}{2}}
E_i^{(n)} E_jE_i^{(n')}=0
\]
\[
\tag{f}
\sum_{n+n'=1-\ang{i,j'}}(-1)^{n'}\pi_i^{n'p(j)+\binom{n'}{2}}
F_i^{(n)}F_jF_i^{(n')}=0
\]
where for any element $\nu=\sum_i \nu_i i\in \Z[I]$ we have set
$\wtd{K}_\nu=\prod_i K_{d_i\nu_i i}$, $\wtd{J}_\nu=\prod_i
J_{d_i\nu_i i}$. In particular, $\wtd{K}_i=K_{d_i i}$,
$\wtd{J}_i=J_{d_i i}$. (Under the bar-consistent condition~(e),
$\wtd{J}_i=1$ for $i\in I_\zero$ while $\wtd{J}_i=J_i$ for $i\in
I_\one$.) 
We endow $\UU$ with a $\Z[I]$-grading $|\cdot |$ by setting
$
|E_i|=i,\quad |F_i|=-i,\quad |J_\mu|=|K_\mu|=0.$
The parity on $\UU$ is given by $p(E_i)=p(F_i)=p(i)$ and $p(K_\mu)=p(J_\mu)=0$,

The algebra $\bold{U}$ has an ${\A}^{\pi}$-form ${}_{{\A}}\bold{U}$. By a base change, we obtain ${}_{R}\bold{U} = R^{\pi} \otimes_{{\A}^{\pi}} {}_{{\A}}\bold{U}.$ Let ${}_{R}\bold{U}^{+}$ (resp. ${}_{R}\bold{U}^{-}$) denote the subalgebra of ${}_{R}\bold{U}$ generated by the $E_{i}^{(n)} = E_{i}^{n}/[n]_{\qq_{i}, \pi_{i}}^{!}$ (resp. $F_{i} =F_{i}^{n}/[n]_{\qq_{i}, \pi_{i}}^{!}$). As a $R^{\pi}$-algebra ${}_{R}\bold{f}$ is isomorphic to ${}_{R}\bold{U}^{+}$ (resp. ${}_{R}\bold{U}^{-}$)  via the map $x \mapsto x^{+}$ (resp. $x \mapsto x^{-}$), where $(\theta_{i}^{(n)})^{+} = E_{i}^{(n)}$ (resp. $(\theta_{i}^{(n)})^{-} = F_{i}^{(n)}.$ 

Denote by $X^+=\{\la\in X\mid \langle i, \la \rangle \in \N, \forall i\in I\}$, the set of dominant integral weights. 

For $\la \in X$, let $M(\lambda) $ be the Verma module of $\bold{U}$, and we can naturally identify $M(\la) = \bold{f}$ as $\Q(q)^\pi$-modules. The ${}_{{\A}}\bold{U}$-submodule ${}_\A M(\lambda)$ can be identified with ${}_{{\A}}\bold{f}$ as $\A^\pi$-free modules. For $\la \in X^+$, we define the integrable $\bold{U}$-module $V(\lambda) =M(\la)/J_\la$, where $J_\la$ is the left $\bold{f}$-module generated by $\theta_i^{\langle i, \la \rangle +1}$ for all $i\in I$. Let $ {}_R M(\lambda) = R^\pi \otimes_{\A^\pi} {}_\A M(\lambda)$ for $\la \in X$, and $ {}_R V(\lambda) = R^\pi \otimes_{\A^\pi} {}_\A V(\lambda)$ for $\la \in X^+$.

The algebra $\bold{U}^{\diamond}$ is defined in the same way as $\bold{U}$ based on the root datum $(Y^{\diamond}, X^{\diamond}, . . .).$

Recall from \cite[Definition 4.2]{CFLW} that the modified quantum covering group $\dot{\bold{U}}$ is a $\mathbb{Q}(q)^{\pi}$-algebra without unit which is
generated by the symbols $1_{\lambda}, E_i1_{\lambda}$ and $F_i1_{\lambda}$, for $\lambda\in X$
and $i \in I$, subject to the relations:
\begin{eqnarray*}
 &1_{\lambda} 1_{\lambda'} =\delta_{\lambda, \lambda'} 1_{\lambda},
   \vspace{6pt}\label{eq:modified idemp rel}\\
&   (E_i1_{ \lambda })  1_{\lambda'} =  \delta_{\lambda, \lambda'} E_i1_{ \lambda}, \quad
1_{\lambda'}   (E_i1_{\lambda}) = \delta_{\lambda', \lambda+ i'} E_i1_{\lambda},
     \vspace{6pt}\label{eq:modified E rel}\\
& (F_i1_{ \lambda})  1_{\lambda'}=  \delta_{\lambda, \lambda'} F_i1_{ \lambda}, \quad
1_{\lambda'}  (F_i1_{\lambda})=  \delta_{\lambda', \lambda- i'} F_i1_{\lambda},
\vspace{6pt}\label{eq:modified F rel}\\
 & (E_iF_j-\pi^{p(i)p(j)}F_jE_i)1_{\lambda}=\delta_{ij}\bra{\langle i, \lambda\rangle}_{v_i,\pi_i}1_\lambda,
 \vspace{6pt}
 \label{eq:modified comm rel}\\
 & \sum_{n+n'=1-\ang{i,j'}}(-1)^{n'}\pi_i^{n'p(j)+\binom{n'}{2}} E_i^{(n)} E_jE_i^{(n')} 1_{\lambda}=0
 \;\; (i\neq j),  
 \vspace{6pt}
   \label{eq:modified E Serre}\\
 &\sum_{n+n'=1-\ang{i,j'}}(-1)^{n'}\pi_i^{n'p(j)+\binom{n'}{2}} F_i^{(n)}F_jF_i^{(n')} 1_{\lambda}=0
  \;\;  (i\neq j),  
   \label{eq:modified F Serre}
\end{eqnarray*}
where $i,j \in I$, $\la, \la'\in X$, and we use the notation $xy1_\lambda=(x1_{\lambda+|y|})(y1_\lambda)$
for $x,y\in \UU$.

The modified quantum covering group $\dot{\bold{U}}$ admits an ${\A}^{\pi}$-form, ${}_{{\A}}\dot{\bold{U}}$ and so we can define ${}_{R}\dot{\bold{U}} = R^{\pi} \otimes_{{\A}^{\pi}} {}_{{\A}}\dot{\bold{U}}$. Let us give a presentation for ${}_{R}\dot{\bold{U}}$. 
\begin{lem}
\label{lem:31.1.3}
The modified quantum covering group ${}_{R}\dot{\bold{U}}$ is generated as an $R^{\pi}$-algebra by $x^{+}\bold{1}_{\lambda}x'^{-}$ or equivalently by $x^{-}\bold{1}_{\lambda}x'^{+}$, where $x \in \fR_{\mu}, x' \in \fR_{\nu}$ and $\lambda \in X$, subject to the following relations:
\begin{align*}
(\theta_{i}^{(N)})^{+}& \bold{1}_{\lambda}(\theta_{i}^{(M)})^{-} 
\\
=& \sum_{t \geq 0} \pi_{i}^{MN - \binom{t + 1}{2}}(\theta_{i}^{(M-t)})^{-}\bbinom{M + N + \langle i, \lambda \rangle}{t}_{\qq_{i}, \pi_{i}} \bold{1}_{\lambda +(M+N-t)i'}(\theta_{i}^{(N-t)})^{+},
\\
(\theta_{i}^{(N)})^{-}& \bold{1}_{\lambda}(\theta_{i}^{(M)})^{+} 
\\
=& \sum_{t \geq 0} \pi_{i}^{MN + t\langle i, \lambda \rangle - \binom{t}{2}}(\theta_{i}^{(M-t)})^{+}\bbinom{M + N - \langle i, \lambda \rangle}{t}_{\qq_{i}, \pi_{i}} \bold{1}_{\lambda -(M+N-t)i'}(\theta_{i}^{(N-t)})^{-},
\end{align*}
\[
(\theta_{i}^{(N)})^{+}(\theta_{j}^{(M)})^{-}\bold{1}_{\lambda} = \pi^{MNp(i)p(j)}(\theta_{j}^{(M)})^{-}(\theta_{i}^{(N)})^{+}\bold{1}_{\lambda}, \text{ for } i \neq j,
\]
\[
x^{+}\bold{1}_{\lambda} = \bold{1}_{\lambda + \mu}x^{+}, \quad x^{-}\bold{1}_{\lambda} = \bold{1}_{\lambda - \mu}x^{-},
\]
\[
(x^{+}\bold{1}_{\lambda})(\bold{1}_{\lambda'}x'^{-}) = \delta_{\lambda,\lambda'}x^{+}\bold{1}_{\lambda}x'^{-}, \quad (x^{-}\bold{1}_{\lambda})(\bold{1}_{\lambda'}x'^{+}) = \delta_{\lambda,\lambda'}x^{-}\bold{1}_{\lambda}x'^{+},
\]
\[
(x^{+}\bold{1}_{\lambda})(\bold{1}_{\lambda'}x'^{-}) = \delta_{\lambda,\lambda'}\bold{1}_{\lambda + \mu}x^{+}x'^{-}, \quad (x^{-}\bold{1}_{\lambda})(\bold{1}_{\lambda'}x'^{+}) = \delta_{\lambda,\lambda'}\bold{1}_{\lambda - \mu}x^{-}x'^{+},
\]
\[
(rx + r'x')^{\pm}\bold{1}_{\lambda} = r{x}^{\pm}\bold{1}_{\lambda} + r'x'^{\pm}\bold{1}_{\lambda}, \text{ where } r, r' \in R^{\pi}.
\]
\end{lem}

\begin{proof}
This is proved in the same way as \cite[\S 31.1.3]{Lu94}. Let $A$ be the $R^{\pi}$-algebra with the above generators and relations.  All of these relations are known to hold in ${}_{R}\dot{\bold{U}}$.  The first three are shown to hold in ${}_{R}\dot{\bold{U}}$ by a direct application of \cite[Lemma 2.2.3]{CHW1} as in \cite[Lemma 4]{Cl14} while the remaining ones are clear. However, there was an error in the second relation of \cite[Lemma 4]{Cl14}, so we derive that relation from \cite[Lemma~ 2.2.3]{CHW1} here. We have
\begin{align*}
& (\theta_{i}^{(N)})^{-} \bold{1}_{\lambda}(\theta_{i}^{(M)})^{+} 
\\
&= (\theta_{i}^{(N)})^{-}(\theta_{i}^{(M)})^{+}\bold{1}_{\lambda - Mi'} \\
&= \sum_{t \geq 0} (-1)^{t}\pi_{i}^{(M-t)(N-t) - t^{2}}(\theta_{i}^{(M-t)})^{+}\bbinom{\tilde{K}_{i};  M + N - (t + 1)}{t}_{\qq_{i}, \pi_{i}}(\theta_{i}^{(N-t)})^{-}\bold{1}_{\lambda - Mi'} \\
&= \sum_{t \geq 0} (-1)^{t}\pi_{i}^{(M-t)(N-t) - t^{2}}(\theta_{i}^{(M-t)})^{+}\bbinom{\langle i, \lambda \rangle - M - N + t -1}{t}_{\qq_{i}, \pi_{i}} \bold{1}_{\lambda -(M+N-t)i'}(\theta_{i}^{(N-t)})^{-} \\
&=  \sum_{t \geq 0} \pi_{i}^{MN + t\langle i, \lambda \rangle - \binom{t}{2}}(\theta_{i}^{(M-t)})^{+}\bbinom{M + N - \langle i, \lambda \rangle}{t}_{\qq_{i}, \pi_{i}} \bold{1}_{\lambda -(M+N-t)i'}(\theta_{i}^{(N-t)})^{-}
\end{align*} 
where in the last step, we used \cite[(1.10)]{CHW1} with $a = M + N - \langle i, \lambda \rangle$.
Hence the natural homomorphism $A \longrightarrow {}_{R}\dot{\bold{U}}$ is surjective. Let $\bold{S}$ be an $R^{\pi}$-basis of $\fR$ consisting of weight vectors.  Then $  \{ x^{+}\bold{1}_{\lambda}x'^{-} | x,x' \in \bold{S}, \lambda \in X \} $ can be seen to be an $R^{\pi}$-basis for $A$, and it is known to be one for ${}_{R}\dot{\bold{U}}$ (cf. \cite[Lemma 5]{Cl14}).  Thus, the natural homomorphism is, in fact, an isomorphism.
\end{proof}

\subsection{}

The algebra $\dot{\bold{U}}^{\diamond}$ is defined in the same way using $\bold{U}^{\diamond}$ and $(Y^{\diamond}, X^{\diamond}, . . . ),$ and so it also has an ${\A}^{\pi}$-form ${}_{{\A}}\dot{\bold{U}}^\diamond $ and we can define ${}_{R}\dot{\bold{U}}^\diamond = R^{\pi} \otimes_{{\A}^{\pi}} {}_{{\A}}\dot{\bold{U}}^\diamond $.

\begin{rem}  \label{rem:parityff}
If $\ell$ is even, then $\fRd$ is a (non-super) algebra; if $\ell$ is odd,
then the $\theta_i$ in $\fRd$ and $\fR$ for any given $i$ have the same parity. 
\end{rem}

For $i\in I$, we denote 
\begin{equation}
  \label{eq:diamond}
q_i^\diamond = q^{i\diamond i/2} = q_i^{\ell_i^2},
\qquad
\qq_i^\diamond = \qq^{i\diamond i/2} = \qq_i^{\ell_i^2},
\qquad
\pi_i^\diamond = \pi^{i\diamond i/2} = \pi_i^{\ell_i^2}.
\end{equation}

\begin{lem}
 \label{lem:parity}
Let $i\in I_1$.  
\begin{enumerate}
\item[(a)]
If $\ell$ is odd, then  $\pi_i^\diamond =\pi_i$.

\item[(b)]
If $\ell$ is even, then   $\pi_i^\diamond =1$.
\end{enumerate}
\end{lem}

\begin{proof}
Recall from Lemma~\ref{lem:parity1} that $\ell_i$ must have the same parity as $\ell$. 
The claim on $\pi_i^\diamond$ follows now from \eqref{eq:diamond}. 
\end{proof}

For each $i\in I$, we have 
\begin{equation}
  \label{eq:qc}
\pi_i^\diamond \qq_i^{\diamond 2}= (\pi_i \qq_i^2)^{\ell_i^2} =1.
\end{equation} 
Following Lusztig \cite{Lu94}, we will refer to the quantum supergroup ${}_{R}\fd$ associated to
 $(Y^\diamond, X^\diamond, \cdots)$ as {\em quasi-classical}; cf. \eqref{eq:qc}.

\begin{prop}\label{quasi}
Let $R$ be the fraction field of ${\A}'$. The quasi-classical algebra $\fRd$ is isomorphic to $\tilde{\fRd}$, the $R^\pi$-algebra generated by $\theta_{i}$, $i \in I$, subject to the super Serre relations:
\[
\sum_{n+n'=1-\ang{i,j'}^\diamond}(-1)^{n'}(\pi_i^\diamond)^{np(j)+\binom{n}{2}}
\theta_i^{(n)}\theta_j\theta_i^{(n')}=0  \qquad (i\neq j \in I).
\]
\end{prop}

\begin{proof}
When $\pi_i = 1$ or $\ell$ is even, $\pi_{i}^{\diamond} = 1$ and $\qq_{i}^{\diamond} = \pm 1$ for each $i \in I$.  Hence, in this case the lemma reduces to \cite[\S 33.2]{Lu94}.

Now let $\ell$ be odd and $\pi=-1$.  We make use of the {\em weight-preserving} automorphism $\dot{\Psi}$ of $_{R}\dot{\bold{U}}^{\diamond}$ (called a twistor) given in \cite[Theorem 4.3]{CFLW} when the base ring contains $\sqrt{-1}.$ 
We will only recall the basic property of $\dot{\Psi}$ which we need, and refer to \cite{CFLW} for details. 
Note that for all $i \in I$, $\qq_{i}^{\diamond}$ is a power of $\sqrt{-1}$ with at least one of the $\qq_{i}^{\diamond} = \pm \sqrt{-1}$. Thus, $\pm \sqrt{-1}$ will play the role played by the $v$ in \cite[Theorem 4.3]{CFLW}, which we will denote by $\tilde{v}$ in this proof so as not to confuse it with the $v$ defined in this paper.  Recall $\dot{\Psi}$ takes $\pi$ to $-\pi$ and $\tilde{v}$ to $\sqrt{-1}\tilde{v}$. When we specialize $\pi = -1$ and $\tilde{v} = \pm \sqrt{-1}$, we obtain an $R$-linear isomorphism of that specialization of $_{R}\dot{\bold{U}}^{\diamond}$, denoted by $_{R}\dot{\bold{U}}^{\diamond}|_{-1}$, with the (quasi-classical) modified quantum group corresponding to the specialization $\pi = 1$ and $\qq_{j}^{\diamond}= \pm1$, denoted by $_{R}\dot{\bold{U}}^{\diamond}|_{1}.$ 

Write 

$\triangleright$ ${}_{R_{-1}}\bold{f}$ for the half quantum (super)group over $R$ corresponding to the former (i.e., $\pi=-1$);

$\triangleright$ ${}_{R_{1}}\bold{f}^{\diamond}$ for the half (quasi-classical) quantum group over $R$ corresponding to the latter (i.e., $\pi=1$); cf. \cite[33.2]{Lu94}. 

Recall that $\fRd$ is a direct sum of finite-dimensional weight spaces ${}_R \bold{f}^{\diamond}_{\nu}$, where $\nu \in \mathbb{Z}_{\geq0}[I]$. The weight-preserving isomorphism $\dot{\Psi}$ above implies that 
\[
\dim_{R^\pi} ({}_R \bold{f}^{\diamond}_{\nu}) = \dim_{R} ({}_{R_{-1}} \bold{f}^{\diamond}_{\nu}) =\dim_R {}_{R_{1}} \bold{f}^{\diamond}_{\nu}, \quad \forall \nu.
\] 
As ${}_{R_{1}}\bold{f}^{\diamond}$ is quasi-classical in the sense of \cite[33.2]{Lu94}, we have $\dim_R {}_{R_{1}} \bold{f}^{\diamond}_{\nu}=\dim_R {}_{R_{1}} \bold{f}_{\nu}$ for all $\nu$, by \cite[33.2.2]{Lu94}, where ${}_{R_{1}}\bold{f}$ is the enveloping algebra of the half KM algebra over $R$. Hence we have
\begin{align}  \label{eq:sameD}
\dim_{R^\pi} ({}_R \bold{f}^{\diamond}_{\nu})  = \dim_{R} ({}_{R_{1}} \bold{f}_{\nu}), \quad \forall \nu.
\end{align}

Since the super Serre relations hold in $\fRd$ (cf. \cite[Proposition 1.7.3]{CHW1}) we have a surjective algebra homomorphism $\varphi: \tilde{\fRd} \longrightarrow \fRd$ mapping $\theta_{i} \mapsto \theta_{i}$ for all $i$. Then $\varphi$ maps each weight space ${}_R \bold{\tilde{f}}^{\diamond}_{\nu}$ onto the corresponding weight space ${}_R \bold{f}^{\diamond}_{\nu}$. As $\tilde{\fRd}$ has a Serre-type presentation by definition, it follows by \cite{KKO14, CHW2} that $\dim_{R^\pi} ({}_R \bold{\tilde{f}}_{\nu}) = \dim_{R} ({}_{R_{1}} \bold{f}_{\nu})$
 for each $\nu$. This together with \eqref{eq:sameD} implies that $\dim_{R^\pi} ({}_R \bold{\tilde{f}}_{\nu}) =\dim_{R^\pi} ({}_R \bold{f}^{\diamond}_{\nu})$. 
Therefore $\varphi$ is a linear isomorphism on each weight space and thus an isomorphism.
\end{proof}

\subsection{} Below we provide an analogue of \cite[35.1.5]{Lu94}. 

\begin{lem}
\label{lem:35.1.5}
Assume that both $n\in \Z$ and $t \in \N$ are divisible by $\ell_i$. Then
$$
\bbinom{n}{t}_{\qq_i,\pi_i} = \bbinom{n/\ell_i}{t/\ell_i}_{\qq_i^\diamond, \pi_i^\diamond}.
$$
\end{lem}
(Setting $\pi=1$ in the above formula recovers \cite[35.1.5]{Lu94}.)

\begin{proof}
By Lemma~\ref{lem:34.1.2}(b), we have
$$
\bbinom{n}{t}_{\qq_i,\pi_i} = \pi_i^{t (n-(t-\ell_i)/2)} \qq_i^{t(n+\ell_i)} \binom{n/\ell_i}{t/\ell_i}.
$$
Note that $\pi_i^\diamond {\qq_i^\diamond}^2 = (\pi \qq^2)^{\frac{i\cdot i}2 \ell_i^2}$.
Since $ (\pi \qq^2)^{2\ell} =1$ and $\ell$ divides $\frac{i\cdot i}2 \ell_i^2$  by 
the definition of $\ell_i$, we have $(\pi_i^\diamond {\qq_i^\diamond}^2)^2 =1$.
Hence by \eqref{eq:diamond} and Lemma~\ref{lem:34.1.2}(b) with $\ell=1$ we have
$$
\bbinom{n/\ell_i}{t/\ell_i}_{\qq_i^\diamond, \pi_i^\diamond} 
=\pi_i^{t (n-(t-\ell_i)/2)} \qq_i^{t(n+\ell_i)} \binom{n/\ell_i}{t/\ell_i}.
$$
The lemma follows.
\end{proof}

\section{The Frobenius-Lusztig homomorphism}
  \label{sec:Frob}

In this section we establish the Frobenius-Lusztig homomorphism between the quasi-classical covering group and the quantum covering group at roots of 1. We also formulate Lusztig-Steinberg tensor product theorem in this setting. 

\subsection{}
\label{subsec:ab}

Following \cite[35.1.2]{Lu94}, in this and following sections we shall assume 
\begin{enumerate}
\item[(a)] for any $i\neq j \in I$ with $\ell_j\ge 2$, we have $\ell_i \geq -\langle i,j' \rangle +1$. 

\item[(b)]  $(I,\cdot)$ has no odd cycles. 
\end{enumerate}

\subsection{}
Below is a generalization of \cite[Theorem~35.1.8]{Lu94}.

\begin{thm}
 \label{th:Frob'}
There is a unique $R^\pi$-superalgebra homomorphism 
$$
\Fr': \fRd \longrightarrow \fR,
\qquad\quad   
\Fr' (\theta_i^{(n)}) = \theta_i^{(n\ell_i)} \quad (\forall i\in I, n\in \Z_{>0}).
$$ 
\end{thm}
(Be aware that the two $\theta_i$'s above belong to different algebras and hence are different.
Theorem~\ref{th:Frob'} is consistent with Remark~\ref{rem:parityff}.)

The rest of the section is devoted to a proof of Theorem~\ref{th:Frob'}. The same remark as in \cite[35.1.11]{Lu94}
allows us to reduce the proof to the case when $R$ is the quotient field of $\A'$, which we will assume
in the remainder of this and the next section. 

\subsection{}

Recall from \eqref{eq:order} that $\pi^\ell \qq^{2\ell}=1$ and $\pi^t \qq^{2t} \neq 1$ for $0<t<\ell$. By the definition of $\ell_i$, we have 
$\pi_i^\ell \qq_i^{2\ell}=1$ and $\pi_i^t \qq_i^{2t} \neq 1$ for $0<t<\ell_i$. Then $[t]_{\qq_i}^\pi!$ is invertible in $R^\pi$, for $0<t <\ell_i$.

The following is an analogue of \cite[Lemma~35.2.2]{Lu94} and the proof uses now Lemmas~\ref{lem:34.1.2} and \ref{lem:34.1.3}. 

\begin{lem}
  \label{lem:fRgen}
The $R^\pi$-superalgebra $\fR$ is generated by the elements $\theta_i^{(\ell_i)}$ for all $i\in I$
and the elements $\theta_i$ for $i\in I$ with $\ell_i \ge 2$.
\end{lem}

\begin{proof}
By definition the algebra $\fR$ is generated by $\theta_i^{(n)}$ for all $i\in I$ and $n \ge 0$.
We can write $n=a+\ell_i b$, for $0\le a <\ell_i$ and $b\in \N$. 
We note the following three identities in $\fR$:
\begin{align}  
\theta_i^{(a+\ell_i b)} &=\qq_i^{\ell_i ab} \tha_i^{(a)}  \tha_i^{(\ell_i b)},
  \label{id:a}
  \\
\tha_i^{(a)} &= [a]_{\qq_i,\pi_i}^{-1} \tha_i^a,
  \label{id:b}
 \\
\tha_i^{(\ell_i b)} &= (b!)^{-1} (\pi_{i}\qq_i)^{-\ell_i^2 \binom{b}{2}} (\tha_i^{(\ell_i)})^b, 
 \label{id:c}
\end{align}
where \eqref{id:a} follows by Lemma~\ref{lem:34.1.2} and \eqref{id:c} follows by Lemma~\ref{lem:34.1.3}, respectively. 
(Note that a sign in the power of $\vv_i$ in the identity (b) in  \cite[proof of Lemma~35.2.2]{Lu94} is optional, but the sign
cannot be dropped from the power of $\qq_i$ in \eqref{id:c}.)
The lemma follows. 
\end{proof}

\subsection{Proof of Theorem~\ref{th:Frob'}}
The uniqueness is clear.

By Lemma~\ref{lem:34.1.3} (with $\ell=1$), we have
\begin{equation}   \label{eq:fact1}
[n]_{\qq^\diamond_i,\pi^\diamond_i}^{!} =(\pi_{i}\qq_i)^{\ell_i^2 \binom{n}{2}} n!.
\end{equation}
We first observe that the existence of a homomorphism $\Fr'$ such that
$\Fr' (\theta_i) = \theta_i^{(\ell_i)} $ implies that $\Fr' (\theta_i^{(n)}) = \theta_i^{(n\ell_i)} $ for all $n\ge 0$.
Indeed, using \eqref{id:c}-\eqref{eq:fact1} we have
\[
\Fr' (\theta_i^{(n)}) = ([n]_{\qq_i^\diamond,\pi_i^\diamond} !)^{-1} \Fr' (\theta_i)^n 
= \big((\pi_{i}\qq_i)^{\ell_i^2 n(n-1)/2} n! \big)^{-1}  \Fr' (\theta_i)^n 
= \theta_i^{(n\ell_i)}.
\]

Hence it remains to show that there exists an algebra homomorphism $\Fr': \fRd \rightarrow \fR$ such that 
$\tha_i \to \tha_i^{(\ell_i)}, \forall i\in I$. 
By Proposition~\ref{quasi} (also cf. \cite{CHW1}), the algebra $\fRd$ has the following defining relations:
\[
\sum_{n+n'=1-\ang{i,j'}^\diamond}(-1)^{n'} (\pi_i^\diamond)^{np(j)+\binom{n}{2}}
\theta_i^{(n)}\theta_j\theta_i^{(n')}=0  \qquad (i\neq j \in I).
\]
By \eqref{eq:fact1} it suffices to check the following identity in $\fR$: for $i\neq j \in I,$ 
\begin{equation*}
\sum_{n+n'=1-\ang{i,j'} \ell_j/\ell_i}(-1)^{n'}  \pi_i^{\ell_i^2(np(j)+n(n-1)/2)}
(\pi_{i}\qq_i)^{- \ell_i^2 \binom{n}{2}}(\pi_{i} \qq_i)^{- \ell_i^2 \binom{n'}{2}}
\frac{(\theta_i^{(\ell_i)})^{n}}{n!} \theta_j^{(\ell_j)}   \frac{(\theta_i^{\ell_i)})^{n'}}{n'!} 
=0,
\end{equation*}
which,  by the identity \eqref{id:c}, is equivalent to checking the following identity in $\fR$:
\begin{equation}   \label{eq:serreFr'}
\sum_{n+n'=1-\ang{i,j'} \ell_j/\ell_i}(-1)^{n'} \pi_i^{\ell_i^2(np(j)+n(n-1)/2)}
\theta_i^{(\ell_i n)}  \theta_j^{(\ell_j)}   \theta_i^{(\ell_i n')} 
=0.
\end{equation}

It remains to prove \eqref{eq:serreFr'}. Set $\alpha = -\ang{i,j'}$. For any $0\le t \le \ell_i -1$, we set
$$
g_t = \sum_{\stackrel{r,s}{r+s =\ell_j \alpha +\ell_i -t}} 
(-1)^r \pi_i^{\ell_jr p(j) + r(r-1)/2}
q_i^{r(\ell_i-1-t)} \tha_i^{(r)} \tha_j^{(\ell_j)} \tha_i^{(s)} \in \fA.
$$
This is basically $f'_{i,j; \ell_j, \ell_j \alpha +\ell_i -t}$ in \cite[4.1.1(d)]{CHW1} in the notation of $\tha$'s. 
By the higher super Serre relations (see  \cite[Proposition~4.2.4]{CHW1} and \cite[4.1.1(e)]{CHW1}), we have 
$g_t=0$ for all $0\le t \le \ell_i -1$. 
Set 
$$g = \sum_{t=0}^{\ell_i-1} (-1)^t \pi_i^{t(t-1)/2} \qq_i^{\ell_j \alpha t +\ell_i t -t} g_t \tha_i^{(t)},
$$ 
which must be $0$.
On the other hand, setting $s'=s +t$, we have
\begin{equation}  \label{eq:g=0}
(0=)  \; g =\sum_{\stackrel{r, s'}{r+s' =\ell_j\alpha +\ell_i}} c_{r,s'} \tha_i^{(r)} \tha_j^{(\ell_j)} \tha_i^{(s')},
\end{equation}
where
\begin{equation*}
c_{r,s'} = \sum_{t=0}^{\ell_i-1} (-1)^{r+t} \pi_i^{\ell_jr p(j) + r(r-1)/2 +  t(t-1)/2} q_i^{r (\ell_i-1-t) +\ell_j\alpha t +\ell_i t-t} \bbinom{s'}{t}_{q_i,\pi_i}.
\end{equation*}
Taking the image of the identity \eqref{eq:g=0} under the map $\fA \rightarrow \fR$, we have
\begin{equation*}
\sum_{\stackrel{r, s'}{r+s' =\ell_j\alpha +\ell_i}} \phi( c_{r,s'}) \tha_i^{(r)} \tha_j^{(\ell_j)} \tha_i^{(s')} =0 \,\, \in \fR. 
\end{equation*}

For a fixed $s'$, we write $s' =a +\ell_i n$, where $a,n \in \Z$ and $0\le a \le \ell_i -1$. 
Note by Lemma~\ref{lem:34.1.2}(c) that $\bbinom{s'}{t}_{\qq_i, \pi_i} =  \qq_i^{ -\ell_i n t} \bbinom{a}{t}_{\qq_i,\pi_i}$. Now  using $r+s' =\ell_j\alpha +\ell_i$ we compute
\begin{align}
 \phi( c_{r,s'}) 
  &= (-1)^r \qq_i^{r(\ell_i-1)}
   \sum_{t=0}^{\ell_i-1} (-1)^{t} \pi_i^{\ell_jr p(j) + r(r-1)/2 + t(t-1)/2} \qq_i^{t (s'-1)  -\ell_i n t} \bbinom{a}{t}_{\qq_i,\pi_i}
   \notag  \\
  &= (-1)^r \qq_i^{r(\ell_i-1)}
   \sum_{t=0}^a  (-1)^{t} \pi_i^{\ell_jr p(j) + r(r-1)/2 +  t(t-1)/2} \qq_i^{t (a-1)} \bbinom{a}{t}_{\qq_i,\pi_i}
   \notag  \\
   &\stackrel{(a)}{=} \delta_{a,0} (-1)^{\ell_j \alpha +\ell_i -\ell_i n}  \pi_i^{\ell_jr p(j) + r(r-1)/2 } \qq_i^{(\ell_i-1)(\ell_j \alpha +\ell_i -\ell_i n)} 
    \notag  \\
   & \stackrel{(b)}{=} \delta_{a,0} (-1)^{\alpha \ell_j/\ell_i +1-n} \pi_i^{\ell_jr p(j) + r(r-1)/2 - r(\ell_i-1)/2}.
   \label{eq:crs}
\end{align}
The identity (a) above follows by the identity 
$ \sum_{t=0}^a  (-1)^{t} \pi_i^{ t(t-1)/2} \qq_i^{t (a-1)} \bbinom{a}{t}_{\qq_i,\pi_i} =\delta_{a,0}$ (see \cite[1.4.4]{CHW1}),
and (b) follows by the identity $\pi_i^{(\ell_i-1)\ell_i/2} \qq_i^{\ell_i^2-\ell_i} =(-1)^{\ell_i+1}$ (which is an $i$-version of \eqref{eq:34.1.2e}
with the help of $\pi_i^{\ell_i}\qq_i^{2\ell_i} =1$).

Inserting \eqref{eq:crs} into \eqref{eq:g=0} and comparing with \eqref{eq:serreFr'}, we reduce the proof of \eqref{eq:serreFr'} to verifying that
$
\pi_i^{\ell_i^2(np(j)+n(n-1)/2)} = \pi_i^{\ell_j \ell_i n p(j) + \ell_i n(\ell_i n-1)/2 - \ell_i n(\ell_i-1)/2},
$
which is equivalent to verifying
$
\pi_i^{\ell_i^2 np(j)} = \pi_i^{\ell_j \ell_i n p(j)}.$
The latter identity is trivial unless both $i$ and $j$ are  in $I_1$; when both $i$ and $j$ are  in $I_1$, the identity follows from Lemma~\ref{lem:parity1}.
Therefore, we have proved  \eqref{eq:serreFr'} and hence Theorem~\ref{th:Frob'}.


\subsection{}

We develop in this subsection the analogue of \cite[35.3]{Lu94}; recall we are still working under the assumption that $R$ is the quotient field of $ \A' $. 

\begin{prop}
	\label{prop:35.3.2}
Let $ \lambda \in X^\diamond $, i.e., $ \langle i, \lambda \rangle \in \ell_i \Z $ for all $ i \in I $. Let $ M $ denote the simple highest weight module with highest weight $ \lambda $ in the category of $R^\pi$-free weight $\UU$-modules, and let $ \eta $ be a highest weight vector of $ M^\lambda $.	
	\begin{enumerate}
		\item[(a)]
		If $ \zeta \in X $ satisfies $ M^\zeta \neq 0 $, then $ \zeta = \lambda - \sum_i \ell_i n_i i' $, where $ n_i \in \N $. In particular, $ \langle i, \zeta \rangle \in \ell_i \Z $ for all $ i \in I $.	
		\item[(b)]
		If $ i \in I $ is such that $ \ell_i \geq 2 $, then $ E_i, F_i $ act as zero on $ M $.		
		\item[(c)]
		For any $ r \geq 0 $, let $ M'_r $ be the subspace of $ M $ spanned by  $ F_{i_1}^{(\ell_{i_1})} F_{i_2}^{(\ell_{i_2})} \ldots F_{i_r}^{(\ell_{i_r})} \eta $ for various sequences $ i_1, i_2, \ldots, i_r $ in $ I $. Let $ M' = \sum_r M'_r $. Then $ M' = M $.
	\end{enumerate}
\end{prop}

\begin{proof}
	The proof is completely analogous to \cite{Lu94}. All computations are similar except that we are now working over $ R^\pi $ instead of $ R $; and the results follow from Lemma~\ref{lem:34.1.2}, \cite[(4.1) and Proposition 4.2.4]{CHW1}, and Lemma~\ref{lem:fRgen}.
	
	First, we show that
	\item[(d)] $ E_iM'_r = 0 $, $ F_i M'_r = 0 $ for any $ i \in I $ such that $ \ell_i \geq 2 $, 
	\\
which is similarly proved by induction on $ r \geq 0 $. The base case $ r = 0 $ follows from the fact that $ \bbinom{\langle i,\lambda \rangle }{t}_{\qq_i,\pi_i} =0 $ since $ \lambda \in X^{\diamond} $ (using Lemma~\ref{lem:34.1.2}) and the fact that $ E_j^{(n)} F_i \eta $ is an $ R^\pi $-linear combination of $ F_i E_j^{(n)} $ and $ E_j^{(n-1)} $. For the inductive step, we want to show that $ E_i F_j^{(\ell_j)} m = 0 $ and $ F_i F_j^{(\ell_j)} m = 0 $ for any $ i,j \in I $ such that $ \ell_i \geq 2 $ and any $ m \in M'_{r-1} \zeta $. For the first one we use the fact that $ E_i F_j^{(\ell_j)} m $ is an $ R^\pi $-linear combination of $ F_j^{(\ell_j)} E_i m $ and $ F_j^{\ell_j - 1} $ in the case $ \ell_j \geq 2 $, and for $ \ell_j = 1 $ we again use  $ \bbinom{\langle i,\lambda \rangle }{t}_{\qq_i,\pi_i} =0 $ from Lemma~\ref{lem:34.1.2}. For the second one, we may use \cite[(4.1) and Proposition 4.2.4]{CHW1} to write $ F_i F_j^{(\ell_j)} m $ as a $ R^\pi $-linear combination of $ F_j^{(\ell_j - r)} F_i F_j^{(r)} m $ for various $r$ with $ 0 \leq r < \ell_j $, and for such $r$ we have $ F_i F_j^{(r)} m = 0 $ by the induction hypothesis.
	
	Next, we may show by induction on $ r \geq 0 $ that 
	\item[(e)] $ E_i^{(l_i)} M'_r \subset M'_{r-1} $ for any $ i \in I $,
	\\
(by convention $ M'_{-1} = 0 $); again for $ m' \in M'_{r-1} $ we can use the fact that $  E_i^{(l_i)} F_j^{(\ell_j)} m' $ is an $ R^\pi $-linear combination of $ F_j^{(\ell_j)} E_i^{(\ell_i)} m' $ (which is in $ M'_{r-1} $ by the induction hypothesis), and elements of the form $ F_j^{(\ell_j - t )} E_i^{(\ell_i - t)} m' $ with $ t > 0 $ and $ t \leq \ell_i, t \leq \ell_j $ (which as before are zero if $ t < \ell_i $ or if $ t = \ell_i $ and $ t < \ell_j $, by (d), and are in $ M'_{r-1} $ if $ t = \ell_i = \ell_j $).  
	
	The statements (d), (e) together with Lemma~\ref{lem:fRgen} show that $ \sum_r M'_r $ is an $ \UURdot $-submodules of $ M $, and by simplicity of $ M $ it follows that $ M = \sum_r M'_r $, from which (a) and (b) also follow.
\end{proof}

\begin{cor}
	\label{cor:35.3.3}
	There is a unique weight $ \UURdot^{\diamond} $-module structure on $M$ (as in Proposition~\ref{prop:35.3.2}) in which the $ \zeta $-weight space is the same as that in the $ \UURdot^{\diamond} $-modules $M$, for any $ \zeta \in X^{\diamond} \subset X $, and such that $ E_i, F_i \in \fRd $ act as $ E_i^{(\ell_i)}, F_i^{(\ell_i)} \in \fR $. Moreover, this is a simple ($R^\pi$-free) highest weight module for $ \UURdot^{\diamond} $ with highest weight $ \lambda \in X^{\diamond} $.
\end{cor}

\begin{proof}
We define operators $ e_i, f_i : M \to M $ for $ i \in I $ by $ e_i = E_i^{(\ell_i)} $, $ f_i = F_i^{(\ell_i)} $. Using Theorem~\ref{th:Frob'} we see that $ e_i $ and $ f_i $ satisfy the Serre-type relations of $ \fRd $.

If $ \zeta \in X \setminus X^{\diamond} $ we have $ M^\zeta = 0 $ by Proposition~\ref{prop:35.3.2}(a) above. If $ \zeta \in X^{\diamond} $ and $ m \in M^\zeta $, then we have that $ (e_i f_j - f_j e_i)(m) $ is equal to $ \delta_{i,j} \bbinom{\langle i,\lambda \rangle }{\ell_i}_{\qq_i,\pi_i} \cdot m $ plus an $ R^\pi $-linear combination of elements of the form $ F_i^{\ell_i - t} E_i^{\ell_i - t} (m) $ with $ 0 < t < \ell_i $ (this follows by \cite[Lemma 4]{Cl14})
which are zero by Proposition~\ref{prop:35.3.2}(b). Since $ \langle i, \zeta \rangle \in \ell_i \Z $,  we see from Lemma~\ref{lem:35.1.5} that
$$\\
\bbinom{\langle i,\lambda \rangle }{\ell_i}_{\qq_i,\pi_i} = \bbinom{\langle i,\lambda \rangle/\ell_i }{1}_{\qq_i^{\diamond},\pi_i^{\diamond}}
$$
and so $ (e_i f_j - f_j e_i)m = \delta_{i,j} [\langle i,\lambda \rangle/\ell_i ]_{\qq_i^{\diamond},\pi_i^{\diamond}} \cdot m $. We also have that $ e_i(M^\zeta) \subset M^{\zeta + \ell_i i'} $ and $ f_i(M^\zeta) \subset M^{\zeta - \ell_i i'} $. Thus, we have a unital $ \UURdot^\diamond $-module structure on $M$, and by Proposition \ref*{prop:35.3.2}(c) this is a highest weight module of $ \UURdot^\diamond $ with highest weight $ \lambda $ and simplicity also follows using Lemma~\ref{lem:fRgen} in the same argument as in \cite{Lu94}. 
\end{proof}

\subsection{}

Now we are ready to state our analogue of the main result of \cite[35.4]{Lu94} on a tensor product decomposition. Let $\kf$ be the $R$-subalgebra of $\fR$ generated by the elements $ \theta_i $ for various $i$ such that $ \ell_i \geq 2 $. We have $ \kf = \oplus_\nu \kf_\nu $ where $ \kf = \fR_\nu \cap \kf $.

\begin{thm} [Lusztig-Steinberg tensor product theorem]
	\label{thm:tensor}
	The $ R^\pi$-linear map 
	\[
	 \chi : \fRd \otimes_R \kf \to \fR,
	 \qquad
	  x \otimes y \mapsto\Fr'(x) y 
	  \]
 is an isomorphism of $R^\pi$-modules.	
\end{thm}

\begin{proof}
	First, we make the following statement which is similar to (but slightly less precise than) \cite[35.4.2(a)]{Lu94}.
	
{\bf Claim.} For any $ i \in I $ and $ y \in \kf_\nu $, there exists some $ a(y), b(y) \in \Z$ such that the difference $ \theta_i^{(\ell_i)} y - \pi_i^{a(y)} \qq_{i}^{b(y)} y \theta_i^{(\ell_i)} $ belongs to $\kf$. 

For $y=y'y''$ one easily reduces the Claim to the same type of claim for $y'$ and $y''$. Hence it suffices to show this Claim when $y$ is a generator of $ \kf $ i.e. $ y = \theta_j $ where $ \ell_j \geq 2 $. Recall our assumption (a) in \S\ref{subsec:ab} that $\ell_i \geq -\langle i,j' \rangle +1$. Hence, we may use the higher Serre relation in \cite[(4.1) and Proposition 4.2.4]{CHW1} (but with $ \theta_i $'s instead of $ F_i $'s) to show that for some $ a(j), b(j) $, the difference $ \theta_i^{(\ell_i)} \theta_j - \pi_i^{a(j)} \qq_{i}^{b(j)} \theta_j \theta_i^{(\ell_i)} $ is an $ R^\pi $-linear combination of products of the form $ \theta_i^{(r)} \theta_j \theta_i^{(\ell_i - r)} $ with $ 0 < r < \ell_i $, which are contained in $ \kf $ by definition. The Claim is proved. 

By Lemma \ref{lem:fRgen}, $\fR$ is generated by $\theta_i^{(\ell_i)}$ and $\theta_j$ with $\ell_j\ge 2$. The surjectivity of $ \chi $ follows as the Claim allows us to move factors $\theta_j$ to the right which produces lower terms in $\kf$.	
	
The injectivity is proved by exactly the same argument as in \cite[35.4.2]{Lu94} using now Proposition~\ref{prop:35.3.2} and Corollary~\ref{cor:35.3.3}; the details will be skipped.
\end{proof}

The following is an analogue of \cite[Proposition 35.4.4]{Lu94}, which follows by the same argument now using the anti-involution $\sigma$ of $\fR$ which fixes each $\theta_i$ (cf. \cite[\S 1.4]{CHW1}). We omit the detail to avoid much repetition. 
\begin{prop} 
	\label{prop:35.4.4}
	Assume that the root datum is simply connected. Then, there is a unique $ \lambda \in X^+ $ such that $ \langle i, \lambda \rangle = \ell_i - 1 $ for all $i$. Let $ \eta $ be the canonical generator of $ {}_R V(\lambda) $.
	The map $ x \mapsto x^-\eta $ is an $R^\pi$-linear isomorphism $ \kf \longrightarrow  {}_R V(\lambda) $.
\end{prop}

\subsection{}

The following is a generalization of \cite[Theorem 35.1.7]{Lu94}.  As with Theorem \ref{th:Frob'}, we may reduce the proof to the case when $R$ is the quotient field of $\mathcal{A}'$ (cf. \cite[35.1.11]{Lu94}).

\begin{thm}
 \label{th:Frob}
There is a unique $R^\pi$-superalgebra homomorphism 
$
\Fr: \fR \longrightarrow \fRd
$
 such that, for all $i\in I, n\in \N$,
\begin{align*}
\Fr (\theta_i^{(n)}) = 
\begin{cases}
\theta_i^{(n/\ell_i)} , & \text{ if }  \ell_i \text{ divides } n,
 \\
0, & \text{ otherwise}. 
\end{cases}
\end{align*}
\end{thm}
(We call $\Fr$ the Frobenius-Lustig homomorphism.)

\begin{proof}
The proof proceeds essentially like that of \cite[Theorem 35.1.7]{Lu94}. Uniqueness is clear; we need only prove the existence. By Theorem \ref{thm:tensor}, there is an $R^{\pi}$-linear map $P:\fR \longrightarrow \fRd,$ such that for all $i_{k} \in I$ and for $j_{p} \in I$ where $\ell_{j_{p}} \geq 2$
\begin{align*}
P(\theta_{i_{1}}^{(\ell_{i_{1}})}. . .\theta_{i_{n}}^{(\ell_{i_{n}})}\theta_{j_{1}}. . .\theta_{j_{r}}) =
\begin{cases}
 \theta_{i_{1}}. . .\theta_{i_{n}}, &\text{ if }  r = 0,
 \\
0,  &\text{ otherwise}. 
\end{cases}
\end{align*}
We now check that $P$ is a homomorphism of $R^{\pi}$-algebras.  Because $\fR$ is generated as an $R^{\pi}$-module by elements of the form $x = \theta_{i_{1}}^{(\ell_{i_{1}})}. . .\theta_{i_{n}}^{(\ell_{i_{n}})}\theta_{j_{1}}. . .\theta_{j_{r}}$, we need to check that for any such $x,$
\begin{equation}
\label{thetaj}
P(x\theta_{j}) = P(x)P(\theta_{j})
\end{equation}
for $j \in I$ such that $\ell_{j} \geq 2$ and
\begin{equation}
\label{thetai}
P(x\theta_{i}^{(\ell_{i})}) = P(x)P(\theta_{i}^{(\ell_{i})})
\end{equation}
for all $i \in I$. As \eqref{thetaj} is obvious, we will concern ourselves with \eqref{thetai}.  Note that \eqref{thetai} is clear when $r = 0$. Assume now $r>0$. Let us write $x' = \theta_{i_{1}}^{(\ell_{i_{1}})}. . .\theta_{i_{n}}^{(\ell_{i_{n}})}\theta_{j_{1}}. . .\theta_{j_{r-1}}$ and $\theta_{j} = \theta_{j_{r}}$ so that $x = x'\theta_{j}.$ For $ i = j$, we have $P(x)P(\theta_{i}^{(\ell_{i})}) = 0$ and
\begin{align*}
P(x\theta_{i}^{(\ell_{i})}) &= P(x'\theta_{i}\theta_{i}^{(\ell_{i})}) 
= P(x'\theta_{i}^{(\ell_{i})}\theta_{i}) 
= P(x'\theta_{i}^{(\ell_{i})})P(\theta_{i}) 
= 0,
\end{align*}
where the third equality is due to \eqref{thetaj}.  Now suppose that $i \neq j.$ As $\ell_{i} > -\langle i, j' \rangle,$ we may use the higher order Serre relations for quantum covering groups (cf. \cite[(4.1) and Proposition 4.2.4]{CHW1}) to write $\theta_{j}\theta_{i}^{(\ell_{i})}$ as a linear combination of terms of the form $\theta_{i}^{(m)}\theta_{j}\theta_{i}^{(n)}$ where $m + n = \ell_{i}$ and $m \geq 1.$  Because of \eqref{id:b} and \eqref{thetaj}, $P(x'\theta_{i}^{(m)}\theta_{j}\theta_{i}^{(n)}) = 0$ for $1 \leq m < \ell_{i},$ and $P(x'\theta_{i}^{(\ell_{i})}\theta_{j}) = 0.$

Now that we know that $P$ is an $R^{\pi}$-algebra homomorphism, it remains to compute $P(\theta_{i}^{(n)})$ for all $n \in \mathbb{Z}_{\geq0}.$  Write $n = b\ell_{i} + a,$ where $0 \leq a < \ell_{i}$ and $b \in \mathbb{Z}_{\geq0}.$ Using \eqref{id:a}, \eqref{id:b} and \eqref{id:c}, for $a > 0$ we have
\begin{align*}
P(\theta^{(b\ell_{i} + a)}) &= \qq_{i}^{\ell_{i}ab}P(\theta_{i}^{(a)})P(\theta_{i}^{(b \ell_i)}) 
= \qq_{i}^{\ell_{i}ab}([a]_{\qq_{i},\pi_{i}}^{!})^{-1}P(\theta_{i}^{a})P(\theta_{i}^{(b \ell_i)}) 
= 0.
\end{align*}
Similarly, for $a = 0$ we have
\begin{align*}
P(\theta_{i}^{(b\ell_{i})}) &= (b!)^{-1}(\pi_{i}\qq_{i})^{-\ell_{i}^{2}\binom{b}{2}}P(\theta_{i}^{(\ell_{i})})^{b} \\
&= (b!)^{-1}(\pi_{i}^{\diamond}\qq_{i}^\diamond)^{-\binom{b}{2}}\theta_{i}^{b} 
= ([b]_{\qq_{i}^{\diamond}, \pi_{i}^{\diamond}}^{!})^{-1}\theta_{i}^{b} 
= \theta_{i}^{(b)},
\end{align*}
where, in the third equality we used Lemma \ref{lem:34.1.3}, with $\ell =1.$ Hence, $P$ is the desired homomorphism $\Fr$.
\end{proof}

\subsection{}

We extend the Frobenius-Lusztig homomorphism $\Fr: \fR \longrightarrow \fRd$ in Theorem~\ref{th:Frob} to ${}_{R}\dot{\bold{U}}$. In contrast to the quantum group setting, we have to twist $\Fr$ slightly on one half of the quantum covering group. 

\begin{thm}
 \label{th:Frob2}
There is a unique $R^\pi$-superalgebra homomorphism 
$
\Fr: {}_{R}\dot{\bold{U}} \longrightarrow {}_{R}\dot{\bold{U}}^{\diamond}
$
 such that  for all $i\in I, n\in \Z, \lambda \in X$,
\begin{align}
  \label{eq:FrE}
\Fr (E_i^{(n)}\bold{1}_{\lambda}) = 
\begin{cases}
\pi_{i}^{\binom{\ell_{i}}{2}n/\ell_{i}}E_i^{(n/\ell_i)}\bold{1}_{\lambda} , & \text{ if   $\ell_i$ divides  $n$ and ${\lambda} \in X^{\diamond}$},
 \\
0, & \text{ otherwise}
\end{cases}
\end{align}
and
\begin{align*}
\Fr (F_i^{(n)}\bold{1}_{\lambda}) = 
\begin{cases}
F_i^{(n/\ell_i)}\bold{1}_{\lambda} , & \text{ if   $\ell_i$ divides  $n$ and ${\lambda} \in X^{\diamond}$},
 \\
0, & \text{ otherwise}. 
\end{cases}
\end{align*}
\end{thm}
(We also call $\Fr$ in this theorem the Frobenius-Lustig homomorphism.)

\begin{proof}
Let $\Fr: \fR \longrightarrow \fRd$ be the homomorphism from Theorem \ref{th:Frob}. Consider the homomorphism $\tilde{\Fr} = \psi \circ \Fr$, where $\psi: \fRd \longrightarrow \fRd$ is the algebra automorphism such that $\theta_{i}^{(n)} \mapsto \pi_{i}^{n}\theta_{i}^{(n)}.$ The proof, much like that of \cite[Theorem 35.1.9]{Lu94}, amounts to checking that  for $x,x' \in \fR$ the assignment
\[
x^{+}\bold{1}_{\lambda}x'^{-} \mapsto \tilde{\Fr}(x^{+})\bold{1}_{\lambda}Fr(x'^{-}), \quad x^{-}\bold{1}_{\lambda}x'^{+} \mapsto \Fr(x^{-})\bold{1}_{\lambda}\tilde{\Fr}(x'^{+}),
\]
for $ \lambda \in X^{\diamond}$, and 
\[
x^{+}\bold{1}_{\lambda}x'^{-} \mapsto 0, \quad x^{-}\bold{1}_{\lambda}x'^{+} \mapsto 0,
\]
for $\lambda \in X \backslash X^{\diamond}$ satisfies the the appropriate relations. These are the relations of Lemma \ref{lem:31.1.3} for ${}_{R}\dot{\bold{U}}$ and for ${}_{R}\dot{\bold{U}}^{\diamond}$, using Lemma \ref{lem:35.1.5} to deal with the $(\qq,\pi)$-binomial coefficients. The use of the homomorphism $\tilde{\Fr}$ (in place of $\Fr$) on $\bold{U}^+$ is necessitated by the first and second relations in Lemma \ref{lem:31.1.3}. Both sides of the first relation are mapped to zero by $\Fr$ unless $N, M \in \ell_{i}\mathbb{Z}$ and $\lambda \in X^{\diamond}$, so we focus on this case. Recalling $\qq_{i}^{\diamond}, \pi_{i}^{\diamond}$ from \eqref{eq:diamond}, we have
\begin{align*}
\Fr & \left(\sum_{t \geq 0} \pi_{i}^{MN - \binom{t + 1}{2}}F_{i}^{(M-t)}\bbinom{M + N + \langle i, \lambda \rangle}{t}_{\qq_{i}, \pi_{i}} \bold{1}_{\lambda +(M+N-t)i'}E_{i}^{(N-t)} \right) \\
&= \sum_{t \geq 0} \pi_{i}^{MN - \binom{t + 1}{2}} \Fr(F_{i}^{(M-t)})\bbinom{M + N + \langle i, \lambda \rangle}{t}_{\qq_{i}, \pi_{i}} \bold{1}_{\lambda +(M+N-t)i'}Fr(E_{i}^{(N-t)}) \\
&= \sum_{t \geq 0, t\in \ell_{i}\mathbb{Z}} (\pi_{i}^{\diamond})^{(M/\ell_{i})(N/\ell_{i}) - \binom{t/\ell_{i} + 1}{2}}\pi_{i}^{t/\ell_{i}\binom{\ell_{i}}{2}} F_{i}^{((M-t)/\ell_{i})}\bbinom{(M + N + \langle i, \lambda \rangle)/\ell_{i}}{t/\ell_{i}}_{\qq_{i}^{\diamond}, \pi_{i}^{\diamond}} \\
&\qquad \qquad\qquad \cdot \bold{1}_{\lambda +(M+N-t)i'}\pi_{i}^{(N-t)/\ell_{i}\binom{\ell_{i}}{2}}E_{i}^{((N-t)/\ell_{i})} \\
&= \pi_{i}^{N/\ell_{i}\binom{\ell_{i}}{2}}\sum_{t \geq 0, t\in \ell_{i}\mathbb{Z}} (\pi_{i}^{\diamond})^{(M/\ell_{i})(N/\ell_{i}) - \binom{t/\ell_{i} + 1}{2}}F_{i}^{((M-t)/\ell_{i})}\bbinom{(M + N + \langle i, \lambda \rangle)/\ell_{i}}{t/\ell_{i}}_{\qq_{i}^{\diamond}, \pi_{i}^{\diamond}} \\
&\qquad \qquad\qquad\qquad \cdot \bold{1}_{\lambda +(M+N-t)i'}E_{i}^{((N-t)/\ell_{i})} \\
&=\pi_{i}^{N/\ell_{i}\binom{\ell_{i}}{2}}E_{i}^{(N/\ell_{i})}\bold{1}_{\lambda}F_{i}^{(M/\ell_{i})} \\
&=\Fr \big(E_{i}^{(N)}\bold{1}_{\lambda}F_{i}^{(M)} \big),
\end{align*}
where we have used $\pi_{i}^{ - \binom{t + 1}{2}} =(\pi_{i}^{\diamond})^{- \binom{t/\ell_{i} + 1}{2}}\pi_{i}^{t/\ell_{i}\binom{\ell_{i}}{2}}$ and Lemma~\ref{lem:35.1.5} in the second equality above. 

The verification of the second relation of Lemma \ref{lem:31.1.3} is entirely similar, and the other relations therein are straightforward.
\end{proof}

\section{Small quantum covering groups}
  \label{sec:small}

In this section, we construct and study the small quantum covering groups. 
We take $R^{\pi} = \mathbb{Q}(\qq)^{\pi}$, where $\qq$ is as in \eqref{eq:qq}. 

\subsection{}

Let $_{R}\dot{\mathfrak{u}}$ be the subalgebra of $_{R}\dot{\bold{U}}$ generated by $E_{i}\bold{1}_{\lambda}$ and $F_{i}\bold{1}_{\lambda}$ for all $i \in I$ with $\ell_i \ge 2$ and $\lambda \in X.$ It is clear then, that $_{R}\dot{\mathfrak{u}}$ is spanned by terms of the form $x^{+}\bold{1}_{\lambda}x'^{-}$ where $x, x' \in \mathfrak{f}.$ We follow the construction of \cite[\S 36.2.3]{Lu94} in extending $_{R}\dot{\bold{U}}$ to a new algebra $_{R}\hat{\bold{U}}$. Any element of $_{R}\dot{\bold{U}}$ can be written as a sum of the form
$\sum_{\lambda, \mu \in X} x_{\lambda, \mu}$
where $x_{\lambda, \mu} \in \bold{1}_{\lambda}{}_{R}\dot{\bold{U}}\bold{1}_{\mu}$ is zero for all but finitely many pairs $\lambda, \mu.$ We relax this condition in $_{R}\hat{\bold{U}}$ by allowing such sums to have infinitely many nonzero terms provided that the corresponding $\lambda - \mu$ are contained in a finite subset of $X$.  The algebra structure extends in the obvious way.  We define ${}_{R}\hat{\mathfrak{u}}$ to be the subalgebra of ${}_{R}\hat{\bold{U}}$ with $x_{\lambda, \mu} \in  \bold{1}_{\lambda}{}_{R}\dot{\mathfrak{u}}\bold{1}_{\mu}$.

Let $2\tilde{\ell}$ be the smallest positive integer such that $\qq^{2\tilde{\ell}} = 1.$ Hence, $\tilde{\ell} = 2\ell$ for $\ell$ odd and $\tilde{\ell} = \ell$ for $\ell$ even. We define the cosets
\begin{align}
  \label{eq:coset}
\bold{c}_{\bold{a}} = \{ \lambda \in X \mid \langle i, \lambda \rangle \equiv a_{i}\pmod{2\tilde{\ell}}, \quad \forall i \in I \}, 
\end{align}
for ${\bold a} =(a_i | i\in I)$ with $0 \leq a_{i} \leq 2\tilde{\ell} - 1$.  Note that there are at most $(2\tilde{\ell})^{|I|}$ such cosets and they partition $X$. Moreover, for each coset $\bold{c}$, $\bold{1}_{\bold{c}} := \sum_{\lambda \in \bold{c}} \bold{1}_{\lambda}$ is an element of ${}_{R}\hat{\mathfrak{u}}.$

Let ${}_{R}\mathfrak{u}$ (resp. ${}_{R}\mathfrak{u}'$) be the $R^{\pi}$-submodule of ${}_{R}\hat{\mathfrak{u}}$ generated by the elements $x^{+}\bold{1}_{\bold{c}}x'^{-}$ (resp. $x^{-}\bold{1}_{\bold{c}}x'^{+}$) where $x, x' \in \mathfrak{f}.$ The following is an analogue of \cite[Lemma 36.2.4]{Lu94}.

\begin{lem}
\begin{enumerate}
\item For any $u \in {}_{R}\mathfrak{u}$ and $0 \leq M \leq \ell_{i} - 1$,   $F_{i}^{(M)}u$ lies in ${}_{R}\mathfrak{u}$.
\item We have ${}_{R}\mathfrak{u} = {}_{R}\mathfrak{u}'$, and ${}_{R}\mathfrak{u}$ is a subalgebra of ${}_{R}\hat{\mathfrak{u}}$.
\end{enumerate}
\end{lem}
The algebra ${}_{R}\mathfrak{u}$ is called the {\em small quantum covering group}.

\begin{proof}
We follow the proof in \cite{Lu94}.  We prove the first statement by induction on $p$, where our $u = E_{i_{1}}^{(n_{1})}. . .E_{i_{p}}^{(n_{p})}x'^{-}.$  The result is obvious for $p = 0$, so we now consider $p \geq 1$ and rewrite $u$ as
\[
u = \bold{1}_{\bold{c}'}E_{i_{1}}^{(n_{1})}x_{1}^{+}x'^{-}
\]
where $x_{1} = \theta_{i_{2}}^{(n_{2})}. . .\theta_{i_{p}}^{(n_{p})}.$  When $i \neq i_{1}$, the result is immediate, so we consider $i = i_{1}.$ In that case, using the relations of Lemma \ref{lem:31.1.3}, we have 
\begin{align*}
F_{i}^{(M)}u &= \sum_{\lambda \in \bold{c}'} \sum_{t \leq n_{1}, t \leq M} \pi_{i}^{MN +t\langle i, \lambda \rangle - \binom{t}{2}} \bbinom{n_{1} + M - \langle i, \lambda \rangle}{t}_{\qq_{i},\pi_{i}}
\\
&\qquad\qquad\qquad\qquad \cdot E_{i}^{(a_{1}-t)}\bold{1}_{\lambda - (n_{1} + M -t)i'}F_{i}^{(M-t)}x_{1}^{+}x'^{-}.
\end{align*}

Fix  $\mu \in \bold{c}'$.  Then for any $\lambda \in \bold{c}'$,  $n_{1} + M - \langle i, \lambda \rangle \equiv n_{1} + M - \langle i, \mu \rangle \text{ mod}(\ell_{i})$. Using Lemma ~\ref{lem:34.1.2} and noting that $t < \ell_{i}$, we have that 
\begin{align*}
\bbinom{n_{1} + M - \langle i, \lambda \rangle}{t}_{\qq_{i},\pi_{i}} &=\qq_{i}^{-\ell_{i}t(\langle i, \lambda \rangle - \langle i, \mu \rangle)} \bbinom{n_{1} + M - \langle i, \mu \rangle}{t}_{\qq_{i},\pi_{i}} \\
&= \bbinom{n_{1} + M - \langle i, \mu \rangle}{t}_{\qq_{i},\pi_{i}}, 
\end{align*}
where we used in the second equality the condition that $\langle i, \lambda \rangle - \langle i, \mu \rangle \equiv 0 \text{ mod}(2\tilde{\ell})$.  Hence, $F_{i}^{(M)}u $ is equal to 
\begin{align*}
& \sum_{t \leq n_{1}, t \leq M} \pi_{i}^{MN +t\langle i, \mu \rangle - \binom{t}{2}} \bbinom{n_{1} + M - \langle i, \mu \rangle}{t}_{\qq_{i},\pi_{i}}E_{i}^{(a_{1}-t)}(\sum_{\lambda \in \bold{c}'}\bold{1}_{\lambda - (n_{1} + M -t)i'})F_{i}^{(M-t)}x_{1}^{+}x'^{-} \\
&= \sum_{t \leq n_{1}, t \leq M} \pi_{i}^{MN +t\langle i, \mu \rangle - \binom{t}{2}} \bbinom{n_{1} + M - \langle i, \mu \rangle}{t}_{\qq_{i},\pi_{i}}E_{i}^{(a_{1}-t)}\bold{1}_{\bold{c}''}F_{i}^{(M-t)}x_{1}^{+}x'^{-}, 
\end{align*}
for some other $\bold{c}''.$ Hence, $F_{i}^{(M)}u \in {}_{R}\mathfrak{u}$ by induction.  Finally, the second statement is shown by repeated application of this result as in \cite[Lemma 36.2.4]{Lu94}.
\end{proof}

\subsection{}

Recall there are a comultiplication $ \Delta $ and an antipode $S$ on $ \UU $ as defined in \cite[Lemmas~2.2.1, 2.4.1]{CHW1}. 
Write $ {}_{\lambda}\bold{U}_{\mu}$ for the subspace of ${}_{R}\dot{\bold{U}}$ spanned by elements of the form $\bold{1}_{\lambda}x\bold{1}_{\mu}$, where $x \in {}_{R}\bold{U}$ and write $p_{\lambda, \mu}$ for the canonical projection ${}_{R}\bold{U} \rightarrow {}_{\lambda}\bold{U}_{\mu}$.   As in \cite[23.1.5, 23.1.6]{Lu94}, $\Delta$ and $S$ induce $R^{\pi}$-linear maps 
\[
\Delta_{\lambda, \mu, \lambda', \mu'}: {}_{\lambda + \lambda'}\bold{U}_{\mu + \mu'} \longrightarrow {}_{\lambda}\bold{U}_{\mu} \otimes {}_{\lambda'}\bold{U}_{\mu'}
\]
 given by $\Delta_{\lambda, \mu, \lambda', \mu'}(p_{\lambda + \lambda', \mu + \mu'}(x)) = (p_{\lambda,\mu} \otimes p_{\lambda', \mu'})(\Delta(x))$, for $\lambda, \mu, \lambda', \mu' \in X$, and 
 \[
 \dot{S}: {}_{R}\dot{\bold{U}} \longrightarrow {}_{R}\dot{\bold{U}}
 \]
  defined by $\dot{S}(\bold{1}_{\lambda}x\bold{1}_{\mu}) = \bold{1}_{-\mu}S(x)\bold{1}_{-\lambda}$ for $x \in {}_{R}\bold{U}.$
For example, $\Delta(E_i) = E_i \otimes 1 + \tilde{J}_i \tilde{K}_i \otimes E_i$ in $ {}_R {\UU} $, and hence we obtain 
\[
\Delta_{\lambda-\nu+i', \lambda-\nu, \nu, \nu}(  E_i \bold{1}_{\lambda} ) = p_{\lambda - \nu + i',\lambda - \nu} \otimes p_{\nu,\nu} ( E_i \otimes 1 + \tilde{J}_i \tilde{K}_i \otimes E_i) = E_i \bold{1}_{\lambda-\nu} \otimes \bold{1}_\nu.
\]
This collection of maps is called the comultiplication on $ {}_R \dot{\UU} $, and it can be formally regarded as a single linear map 
\[
\dot{\Delta} =\prod_{\lambda, \mu, \lambda', \mu' \in X} \hat{\Delta}_{\lambda,\mu, \lambda', \mu'}: {}_R \dot{\UU} \longrightarrow \prod_{\lambda, \mu, \lambda', \mu' \in X} 
{}_{\lambda}\bold{U}_{\mu} \otimes {}_{\lambda'}\bold{U}_{\mu'}.
\]
A comultiplication $ \dot{\Delta}^\diamond $ on ${}_{R}\dot{\UU}^{\diamond}$ can be defined in the same way.

\begin{prop}
	The Frobenius-Lusztig homomorphism $\Fr$ is compatible with the comultiplications on ${}_{R}\dot{\UU}$ and ${}_{R}\dot{\UU}^{\diamond}$, i.e., $ \dot{\Delta}^\diamond \circ \Fr = (\Fr \otimes \Fr) \circ \dot{\Delta}$.
\end{prop}
 (In the usual quantum group setting this was noted by \cite[35.1.10]{Lu94}.)
 
\begin{proof}
	It suffices to check on the generators $ E_i^{(n)} \bold{1}_\lambda $ and $ F_i^{(n)} \bold{1}_\lambda $. 
	Let $ n = m \ell_i \in \ell_i \Z $, and recall that $ \Fr (E_i^{(m\ell_i)}\bold{1}_{\lambda}) = 
	\pi_{i}^{\binom{\ell_{i}}{2}m}E_i^{(m)}\bold{1}_{\lambda} $ in ${}_{R}\dot{\UU}^{\diamond}$. Using the formula (above \cite[Proposition 2.2.2]{CHW1})
	\[
	\Delta(E_i^{(m)}) = \sum_{p + r = m} (\pi_i q_i)^{pr} E_i^{(p)} (\tilde{J}_i \tilde{K}_i)^r \otimes E_i^{(r)}
	\]
	we see that the nonzero parts in $ \dot{\Delta}^\diamond (\Fr( E_i^{(m\ell_i)}\bold{1}_{\lambda} )) $ computed via \eqref{eq:FrE} are of the form 
	\[ 
	\pi_{i}^{\binom{\ell_{i}}{2}m} ( \pi_i^\diamond q_i^\diamond )^{(p + \langle i, \nu \rangle^\diamond)r } E_i^{(p)} \bold{1}_\nu \otimes E_i^{(r)} \bold{1}_{\lambda - \nu}, \qquad p + r = m
	\]
	for various $ \nu \in X^\diamond $, which coincides with $ \Fr \otimes \Fr $ applied to terms in $ \dot{\Delta} (E_i^{(m\ell_i)}\bold{1}_{\lambda})) $ of the form 
	\[ 
	( \pi_i q_i )^{(p\ell_i + \langle i, \nu \rangle)(r\ell_i)} E_i^{(p\ell_i)} \bold{1}_\nu \otimes E_i^{(r\ell_i)} \bold{1}_{\lambda - \nu}, \qquad p + r = m,
	\]
where we note there is a factor contributing from \eqref{eq:FrE} which matches up with the previous part thanks to $\pi_{i}^{\binom{\ell_{i}}{2}p+\binom{\ell_{i}}{2}r} = \pi_{i}^{\binom{\ell_{i}}{2}m}$; the remaining terms are zero under $ \Fr \otimes \Fr $ since at least one of the divided powers of $ E_i $ appearing in either tensor factor must be not divisible by $ \ell_i $. 
	
	On the other hand, if $ n $ is not divisible by $ \ell_i $, then the right hand side will also be zero, since all the non-zero parts of $ \dot{\Delta} (E_i^{(n)}\bold{1}_{\lambda})) $ will have a tensor factor containing some divided power of $ E_i $ not divisible by $ \ell_i $. 
	
	A similar verification takes care of $ F_i^{(n)} 1_\lambda $.
\end{proof}

\subsection{}

The maps $\dot\Delta$ and $\dot S$ restrict to maps on ${}_{R}\dot{\mathfrak{u}}$, which extend to $R^{\pi}$-linear maps $\hat{\Delta}$ and $\hat{S}$ on ${}_{R}\hat{\mathfrak{u}}$ in the obvious way. 
Henceforth, when we refer to $\hat{\Delta}$ and $\hat{S}$ we mean the restrictions to ${}_{R}\mathfrak{u}.$  

Additionally, for any basis $\bold B$ of $\mathfrak{f}$ consisting of weight vectors, with unique zero weight element equal to $1$, we define  an $R^{\pi}$-linear map $\hat{e}: {}_{R}\mathfrak{u} \rightarrow R^{\pi}$ by:

\begin{align*}
\hat{e}(rb^{+}b'^{-}\bold{1}_{\bold{c}_{\bold a}}) =
\begin{cases}
r, & \text{if   $b,b' =1$ and $\bold{a} ={\bold{0}}$},
 \\
0, & \text{otherwise}. 
\end{cases}
\end{align*}
where $b, b' \in \bold{B}$, $r \in R^{\pi}$, and $\bold c_{\bold a}$ in \eqref{eq:coset}. 

Define the following elements: 
\begin{align}
  \label{eq:KJ}
K_{i} = \sum_{\lambda \in X} \qq^{\langle i, \lambda \rangle}\bold{1}_{\lambda}, \quad J_{i} = \sum_{\lambda \in X} \pi^{\langle i, \lambda \rangle}\bold{1}_{\lambda}, \quad 1 = \sum_{\lambda \in X} \bold{1}_{\lambda}.
\end{align}

\begin{prop}\quad
\begin{enumerate}
\item
The $R^\pi$-algebra ${}_{R}\mathfrak{u}$ has a generating set  $\{E_{i}, F_{i}\; (\forall i \text{ with } \ell_{i} \geq 2), K_{i}, J_{i} \; (\forall i \in I) \}$.
\item
$({}_{R}\mathfrak{u}, \hat{\Delta}, \hat{e}, \hat{S})$ forms a Hopf superalgebra. 
\end{enumerate}
\end{prop}

\begin{proof}
The elements in \eqref{eq:KJ} can be written as 
\[
K_{i} = \sum_{\bold{c}} \qq_{\bold{c},i}\bold{1}_{\bold{c}}, \quad J_{i} = \sum_{\bold{c}} \pi_{\bold{c},i}\bold{1}_{\bold{c}}, \quad 1 = \sum_{\bold{c}} \bold{1}_{\bold{c}},
\]
where we have defined $\qq_{\bold{c},i} = \qq^{\langle i, \lambda \rangle}$ and $\pi_{\bold{c},i} = \pi^{\langle i, \lambda \rangle}$ for any $\lambda \in \bold{c}.$
This implies that these elements are also in ${}_{R}\mathfrak{u}.$
Moreover, we have
\[
\bold{1}_{\bold{c}} = \prod_{i \in I}(2\tilde{\ell})^{-1}(1 + \pi_{\bold{c}, i}J_{i})(1 +\qq_{\bold{c},i}^{-1}K_{i} +\qq_{\bold{c},i}^{-2}K_{i}^{2} +. . .+ \qq_{\bold{c},i}^{1 - \tilde{\ell}}K_{i}^{\tilde{\ell}-1}).
\]
This proves (1). 

A direct computation using these generators shows that $\hat{\Delta}$, $\hat{e}$ and $\hat{S}$ are given by the same formulas as $\Delta$, $e$ and $S$, the former maps inherit the following properties of the latter: $\hat{\Delta}$ is a homomorphism which satisfies the coassociativity  (cf. \cite[Lemmas 2.2.1 and 2.2.3]{CHW1}), $\hat{e}$ is a homomorphism (cf. \cite[Lemma 2.2.3]{CHW1}), and $\hat{S}(xy) = \pi^{p(x)p(y)}\hat{S}(y)\hat{S}(x)$ (cf. \cite[Lemma 2.4.1]{CHW1}). Moreover, the image of $\hat{\Delta}$ (respectively, $\hat{S}$) lies in ${}_{R}\mathfrak{u} \otimes {}_{R}\mathfrak{u}$ (respectively, ${}_{R}\mathfrak{u}$).  Hence (2) holds. 
\end{proof}

\subsection{}

We consider the Cartan datum associated to the Lie superalgebra $\osp(1|2n)$, where $n = |I|$, with the following Dynkin diagram:
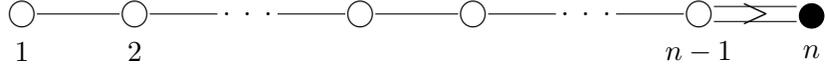
\begin{center}
\begin{tikzpicture}
\node at (0,0) {$\bigcirc$};
\draw (.2,0)--(1.3,0);
\node at (1.5,0) {$\bigcirc$};
\draw (1.7,0)--(2.6,0);
\node at (3,0) {. . .};
\draw (3.4,0)--(4.3,0);
\node at (4.5,0) {$\bigcirc$};
\draw (4.7,0)--(5.8,0);
\node at (6,0) {$\bigcirc$};
\draw (6.2,0)--(7.1,0);
\node at (7.5,0) {. . .};
\draw (7.9,0)--(8.8,0);
\node at (9,0) {$\bigcirc$};
\draw (9.2,-.1)--(10.3,-.1);
\node at (9.75,0) {\Large $>$};
\draw (9.2,.1)--(10.3,.1);
\node at (10.5,-0.05) {\Huge $\bullet$};
\node at (0,-.5) {$1$};
\node at (1.5,-.5) {$2$};
\node at (9,-.5) {$n-1$};
\node at (10.5,-.5) {$n$};
\end{tikzpicture}
\end{center}
The black node denotes the (only) odd simple root. We set
\begin{align*}
i \cdot i = 
\begin{cases}
2, & \text{if   $i$ is odd},
 \\
4, & \text{if $i$ is even}. 
\end{cases}
\end{align*}
The above Cartan datum on $I$ is a super Cartan datum satisfying the bar-consistent condition in the sense of \S \ref{subsec:data}.

\begin{prop}
The small quantum covering group ${}_{R}\mathfrak{u}$ of type $\osp(1|2n)$ is a finite dimensional $R^{\pi}$-module. In particular,
\begin{align*}
\dim_{R^{\pi}}({}_{R}\mathfrak{u}) 
= \dfrac{\ell^{2n^2}}{\gcd(2, \ell)^{2n^{2} - 2n}} (2\tilde\ell)^{n} =
\begin{cases} 
  \ell^{2n^{2}}(4\ell)^{n}, & \text{for $\ell$ odd},  \\
 \dfrac{\ell^{2n^2}}{2^{2n^{2} - 2n}} (2\ell)^{n}, & \text{for $\ell$ even,}
\end{cases}
\end{align*}
when $X$ is the weight lattice, and similarly,
\begin{align*}
\dim_{R^{\pi}}({}_{R}\mathfrak{u}) 
= \dfrac{\ell^{2n^2}}{\gcd(2, \ell)^{2n^{2} - 2n}} 2^{n-1}\tilde\ell^{n} =
\begin{cases} 
  \ell^{2n^{2}}2^{2n-1}\ell^{n}, & \text{for $\ell$ odd},  \\
 \dfrac{\ell^{2n^2}}{2^{2n^{2} - 2n}} 2^{n-1}\ell^{n}, & \text{for $\ell$ even,}
\end{cases}
\end{align*}
when $X$ is the root lattice.
\end{prop}

\begin{proof}
Note that ${}_{R}\mathfrak{u}$ is a $\mathfrak{f} \otimes \mathfrak{f}^{\text{opp}}$ module with basis given by the $\bold{1}_{\bold{c}}$ defined above. This basis has at most $(2\tilde{\ell})^{n}$ elements for any $X$. In particular, it has $(2\tilde{\ell})^{n}$ elements when $X$ is the weight lattice, and $2^{n-1}\tilde{\ell}^{n}$ elements when $X$ is the root lattice, as the root lattice is index $2$ in the weight lattice. Moreover, by Proposition \ref{prop:35.4.4}, we have that $\dim_{R^{\pi}}(\mathfrak{f}^{\pm}) = \dim_{R^{\pi}}({}_{R}V(\lambda))$, where $\lambda$ is the unique weight such that $\langle i, \lambda \rangle = \ell_{i} - 1$ for each $ i \in I.$ Let $V(\lambda)_{1}$ (respectively, $V(\lambda)_{-1}$) be the quotient of the Verma module of highest weight $\lambda$ by its maximal ideal for the quantum group (resp. quantum supergroup) to which the quantum covering group specializes at $\pi = 1$ (respectively, $\pi = -1$) with base field $R = \mathbb{Q}(\vep)$ (recall from \S\ref{subsec:root1} that $\vep$ is an $\ell'$-th root of unity). Because 
\[
{}_{R}V(\lambda) = (\pi + 1){}_{R}V(\lambda) \oplus (\pi - 1){}_{R}V(\lambda) \cong V(\lambda)_{1} \oplus V(\lambda)_{-1}
\]
 and the characters of $V(\lambda)_{1}$ and $V(\lambda)_{-1}$ coincide for dominant weights (cf. \cite{KKO14}, \cite[Remark 2.5]{CHW2}), we have 
\begin{align*}
\dim_{R^{\pi}} \mathfrak{f}^{\pm} &=\dim_{R^\pi}\; {}_{R}V(\lambda)
= \dim_{R} V(\lambda)_{1} = \dim_{R}  \mathfrak{f}_{1}^{\pm}
= \dfrac{\ell^{n^2}}{\gcd(2, \ell)^{n^{2} - n}}
\end{align*}
where $\mathfrak{f}_{1}$ is the (non-super) half small quantum group, i.e., $\mathfrak{f}$ specialized at $\pi = 1$. The last equality is due to \cite[Theorem 8.3(iv)]{Lu90b}.

\end{proof}

\end{document}